\documentclass[amsthm]{elsart}

\usepackage{yjsco}
\usepackage{natbib}
\usepackage{amsthm}
\usepackage{enumerate}
\usepackage{amsmath}
\usepackage{amssymb}
\usepackage{amsfonts}
\usepackage{epsfig}
\usepackage[matrix,arrow,curve]{xy}
\usepackage{color,graphicx}

\usepackage[breaklinks,bookmarksopen,bookmarksnumbered]{hyperref}
\hypersetup{colorlinks=true,backref=true,citecolor=blue,}

\newcommand{\mcp}{\mathcal{P}}

\newcommand{\C}{\mathcal{C}}
\newcommand{\E}{\mathcal{E}}

\newcommand{\sO}{\mathcal{O}}
\newcommand{\F}{\mathcal{F}}

\newcommand{\G}{\mathcal{G}}
\newcommand{\HH}{\mathcal{H}}

\newcommand{\rp}{\mathbb{R}}
\newcommand{\fp}{\mathbb{F}}
\newcommand{\cp}{\mathbb{C}}
\newcommand{\qp}{\mathbb{Q}}

\newcommand{\mcpp}{\mathbb{P}}

\newcommand{\np}{\mathbb{N}}

\newcommand{\lp}{\mathbb{L}}
\newcommand{\zp}{\mathbb{Z}}
\newcommand{\kp}{\mathbb{K}}
\newcommand{\ap}{\mathbb{A}}

\newcommand{\wt}{\widetilde}

\DeclareMathOperator{\prim}{prim}

\DeclareMathOperator{\ev}{ev}

\DeclareMathOperator{\lc}{lc}
\DeclareMathOperator{\Char}{Char}

\DeclareMathOperator{\coef}{coeff}

\DeclareMathOperator{\Res}{Res}
\DeclareMathOperator{\Disc}{Disc}
\DeclareMathOperator{\val}{val}

\DeclareMathOperator{\Aut}{Aut}

\DeclareMathOperator{\Card}{Card}
\DeclareMathOperator{\Vol}{Vol}

\DeclareMathOperator{\Conv}{Conv}

\DeclareMathOperator{\N}{N}

\DeclareMathOperator{\D}{D}
\DeclareMathOperator{\reb}{reb}
\DeclareMathOperator{\Supp}{Supp}
\DeclareMathOperator{\ff}{f}

\theoremstyle{plain}
\newtheorem{theorem}{Theorem}[section]
\newtheorem{lemma}[theorem]{Lemma}
\newtheorem{proposition}[theorem]{Proposition}
\newtheorem{corollary}[theorem]{Corollary}

\theoremstyle{definition}
\newtheorem{remark}[theorem]{Remark}
\newtheorem{definition}[theorem]{Definition}
\newtheorem{example}[theorem]{Example}

\begin{document}
\begin{frontmatter}

\title{Bivariate factorization using a critical fiber}

\author{Martin Weimann} 
\ead{weimann@unicaen.fr}

\address{LMNO, Universit\'e de Caen BP 5186, F 14032 Caen Cedex}


\begin{abstract}
We generalize the classical lifting and recombination scheme for rational and absolute factorization of bivariate polynomials to the case of a critical fiber. We explore different strategies for recombinations of the analytic factors, depending on the complexity of the ramification. We show that working along a critical fiber leads in some cases to a good theoretical complexity, due to the smaller number of analytic factors to recombine. We pay a particular attention to the case of polynomials that are non degenerate with respect to their $P$-adic Newton polytopes. 
\end{abstract}


\end{frontmatter}

\section{Introduction}

Let $F\in \kp[x,y]$ be a square-free bivariate polynomial of bidegree $(d_x,d_y)$ defined over a field $\kp$. The lifting and recombination scheme for bivariate factorization consists to recombine the analytic factors in $\kp[[x]][y]$ of $F$ into the rational factors of $F$ over $\kp$. Up to our knowledge, this approach led to the best theoretical complexity for factoring dense bivariate polynomials, see \citep{L}. However, it has only been developed in the case when the fiber $x=0$ is regular, that is when $F(0,y)$ is separable of degree $d_y$. In this article, we generalize it  to the case of a critical (non regular) fiber, both for rational and absolute factorization issues. A first motivation for this work is that for fields with few elements, a regular fiber might not exist. Although working in a well chosen field extension can solve this problem \citep{gat}, this might have a prohibitive cost \citep{BM}. A second motivation is that a critical fiber brings new combinatorial constraints that might speed up the recombination process. In particular, the number of absolute analytic factors to recombine necessarily decreazes along a critical fiber, due do the presence of ramification. Our main result is the existence of a deterministic algorithm that, given the analytic factors of $F$ up to a certain precision $m$, returns the rational factors of $F$ in small polynomial time in the total degree. While the regular case requires a precision $m=d_x+1$ \citep{L}, polynomials with highly $x$-valuated discriminants might need a higher precision for solving recombinations with linear algebra. However, we show that in positive characteristic, the precision $d_x+1$ is always enough to compute the numbers of rational factors and that in zero characteristic the precision $2d_x$ is enough to test irreducibility. Moreover, we exhibate different combinatorial tricks that allow to solve recombinations with precision $d_x+1$ in many reasonnable situations (Subsection \ref{ss4.4}). The algorithms we develop here are  not intended to compete in general with actual implementations, but we illustrate on some examples that working along a critical fiber improves the complexity at least in some particular cases, especially for polynomials that are non degenerate or locally irreducible along the fiber. The strength of our results depends strongly on the complexity of analytic factorization, an issue we have not studied here. 

\paragraph*{Main result.}

The prime divisors of $F$ in the rings $\kp[[x]][y]$ and $\kp[x,y]$ are respectively called analytic and rational factors. The $n$-truncated analytic factorization of $F$ is the data of the residues modulo $x^{n+1}$ of the irreducible analytic factors of $F$. Although it is a fundamental step of our algorithm, we do not pay attention here to the analytic factorization and we introduce the notation $\C(n)=\C(n,F)$ for the number of arithmetic operations over $\kp$ required for computing the $n$-truncated analytic factorization of a polynomial $F\in \kp[x,y]$. When $x=0$ is a regular fiber, it's well known that $\C(n)\subset \wt{\mathcal{O}}(nd_y)$ thanks to the multifactor Hensel lifting \citep{GG}. In general, analytic factorization is more tricky and $\C(n)$ is expected to be closely related to the complexity of Puiseux series computation. Our main hypothesis on $F$ is the following:
$$
{\rm (H)\qquad  \textit{F \,is \,separable\, with \, respect\, to \,y.}}
$$
We can always reduce to hypothesis $(H)$ after applying a separable factorization algorithm. For fields with at least $d_x(2d_y+1)$ elements, the cost of computing separable factorization is  $\wt{\sO}(d_x d_y^2)$ by Proposition 8 in \citep{L2}. This is negligeable when compared to all complexity results we obtain here. Hence, hypothesis $(H)$ might be restrictive for us only for fields with few elements. We denote by:
\vskip2mm
\noindent
\begin{itemize}
\item $p$ the characteristic of $\kp$.
\item $s$ the number of irreducible analytic factors of $F$ in $\kp(x)[y]$. We thus have $s\le d_y$.
\item $q=\lfloor v/d \rfloor $ the integer part of the quotient of the $x$-adic valuation $v$ of the $y$-discriminant of $F$ with the minimal degree $d$ of the analytic factors. This complexity indicator $q$ will be refined in terms of the resultants and the discriminants of the analytic factors (see Section \ref{S3}). 
\item $\omega$ the universal matrix multiplication exponent ($2\le \omega\le 2.5$).
\end{itemize}
\vskip2mm
\noindent

In all of the sequel, we assume that fast mulplication of polynomials is used. Hence two polynomials in $\kp[y]$ of degrees at most $d$ can be multiplied in softly linear time $\wt{\sO}(d)$.
\vskip2mm
\noindent
\begin{thm} \label{t2} Let $m:=\max(q,d_x+1)$. There exists a deterministic algorithm that, given $F\in \kp[x,y]$ satisfying hypothesis $(H)$, returns its irreducible rational factorization with at most

- $\mathcal{O}(md_ys^{\omega-1})+\C(m)$ arithmetic operations over $\kp$ if $p=0$ or $p>d_x(2d_y-1)$; 

- $\mathcal{O}(kmd_ys^{\omega-1})+\sO(k)\C(m)$ arithmetic operations over $\fp_p$ if $\kp=\fp_{p^k}$. 

\end{thm}

We have  $q\in\mathcal{O}(d_xd_y)$ under hypothesis $(H)$ and $q$ can reach this order of magnitude (Section \ref{S3}, Example \ref{ex1}). If the fiber is regular, then $q=0$ in which case our algorithm specializes to that of Lecerf \citep{L}, with a complexity $\mathcal{O}(d_x d_y^{\omega})$. For fields with at least $2d_y-3$ elements, we can always find a fiber over which $q\le d_x+1$, see Remark \ref{remfiber}. Over such a fiber, we get a complexity $\mathcal{O}(d_x d_y^{\omega})+\C(d_x)$. The only difference with Lecerf's algorithm is that we need to compute the truncated analytic factorization along a critical fiber, a difficulty that is compensed by an expected smaller number $s$ of analytic factors to recombine. It's an open question to know if $\C(d_x)\subset \mathcal{O}(d_x d_y^{\omega})$. 
An important case is that of \textit{non degenerate} polynomials, for which all edge polynomials of $F(x,y-\alpha)$ have simple roots for all $\alpha\in \mcpp^1_{\bar{\kp}}$. In that case, $q\le d_x$  and $s$ is strictly smaller to the total number of lattice points of all edges (see Section \ref{Spolytope} for details). In such a case, the analytic factorization reduces after some well chosen monomial change of variables to the classical Hensel lifting or Newton iteration strategies. A brute force complexity analysis leads in that case to $\C(d_x)\subset \sO(sd_xd_y^2)$ but we strongly believe that this result is not optimal.

\paragraph*{Example.} Suppose given two coprime positive numbers $a$ and $b$ and a field $\kp$ of characteristic zero or greater or equal to $2a+b$. Let 
$$
F(x,y)=(y^a+x^b+y^a x^b)(x^a y^b+1)((y-1)^a+x^b+x^b(y-1)^a)\in \kp[x,y].
$$
Then the curve $C\subset \mcpp^1\times \mcpp^1$ defined by $F$ intersects the line $x=0$ exactly at the points $(0,0), (0,1), (0,\infty)$. The Newton diagrams of $F$ at each of the three points are constituted of a unique segment with only two lattice points. Hence, $F$ is necessarily non degenerate and locally irreducible at each point.  In particular, the analytic factors are coprime modulo $x$, and the Hensel lifting strategy leads to $\C(d_x)=\sO(d_xd_y)$. Hence, the all rational factorization requires $\sO(ab)=\sO(d_xd_y)$ operations over $\kp$ which has to be compared to the classical complexity bounds inherent to the choice of a regular fiber, namely $\sO(d_xd_ys^{\omega-1})$ with $s$ the number of rational places over a regular fiber \cite{L}. Of course, the two complexities will be close as soon as $s$ is small. The difference will be more remarkable in the absolute case for which a regular fiber imposes $s=d_y$ (see here after). In that example, we solve recombinations in the absolute case within $\sO(d_xd_y)$ arithmetic operations over $\kp$, while working over a regular fiber would lead to $\sO(d_xd_y^{\omega})$ operations over $\kp$ \citep{CL}.

\paragraph*{Locally irreducible polynomials.} This example motivates to introduce an important class of polynomials for which our approach leads to a good complexity.  We say that $F$ is locally irreducible along the line $x=0$ (resp. absolutely locally irreducible) if the germs of curves $(C,P)\subset (\mcpp_{\kp}^2,P)$ defined by $F$ are irreducible over $\kp$ (resp. over $\bar{\kp}$) at each rational place $P$ of the line $x=0$, including the place at infinity. For example, $F$ is always locally irreducible along a regular fiber. The previous example is also such a polynomial.

\begin{thm} \label{tirr} 
There exists a deterministic algorithm that, given $F\in \kp[x,y]$ absolutely locally irreducible along $x=0$, returns its irreducible rational factorization with one factorization in $\kp[y]$ of degree at most $d_y$ plus

- $\mathcal{O}(d_xd_ys^{\omega-1})$ arithmetic operations over $\kp$ if $p=0$ or $p>d_x(2d_y-1)$.

- $\mathcal{O}(kd_xd_ys^{\omega-1})$
arithmetic operations over $\fp_p$ if $\kp=\fp_{p^k}$ and $p>d_y$. 

\noindent
In the second case, it's enough to suppose that $F$ is locally irreducible over $\kp$.
\end{thm}

Theorem \ref{tirr} is not a direct application of Theorem \ref{t2} since we can have $F$ locally irreducible with $q\approx d_x d_y$. This is for instance the case when the projective curve defined by $F$ is a rational curve with a unique place along $x=0$ and smooth outside this place. If $F$ is non degenerate, checking local irreducibility has a negligeable cost (Section \ref{Spolytope}). In general, this is more tricky. However, it has to be noticed that Abbhyankar developed in \cite{Ab} an algorithm for testing local irreducibility of a germ of curve that do not require blowing-ups or fractional power series (see also \cite{Cos} for a generalization to positive characteristic). The main ingredient is that of \textit{approximate roots} and the algorithm uses almost only resultant computations. Up to our knowledge, no complexity analysis have been done yet.

\paragraph*{Counting factors and testing irreducibility.}

If we rather pay attention to the number of factors, it turns out that we need a lower truncation order ($\sO(d_x)$ for fields of positive characteristic), leading to a better complexity. We say that $\kp$ is an \textit{absolute field} of $F$ if it contains the field of definition of the irreducible absolute factors of $F$, that is, if rational and absolute factorizations coincide (as in the previous example). 

\begin{thm} \label{t1} \begin{enumerate}
\item Suppose that $p=0$. Then we can test irreducibility of a polynomial $F\in \kp[x,y]$ satisfying hypothesis $(H)$ with one factorization in $\kp[y]$ of degree at most $d_y$ plus 
$$
\mathcal{O}(d_xd_ys^{\omega-1})+\C(2d_x) 
$$
arithmetic operations over $\kp$.
\item Suppose that $p>0$ or that $\kp$ is an absolute field of $F$. We can compute the number of rational factors of $F\in \kp[x,y]$ satisfying hypothesis $(H)$ with one factorization in $\kp[y]$ of degree at most $d_y$ plus 

- $\mathcal{O}(d_xd_ys^{\omega-1})+\C(d_x)$ operations over $\kp$ if $p=0$  or $p>d_x(2d_y-1)$.

- $\mathcal{O}(kd_xd_ys^{\omega-1})+\sO(k)\C(d_x)$ operations over $\fp_p$ if $\kp=\fp_{p^k}$. 
\end{enumerate}
\end{thm}

Note that we need not to suppose that $\kp$ is an absolute field of $F$ in the case of positive characteristic, leading in that case to a much stronger result. Roughly speaking, the underlying reason is that the Frobenius gives an efficient test for that an algebraic number in $\bar{\kp}$ lie in the subfield $\kp$ (Section \ref{S5}).

\paragraph*{Absolute factorization.} Finally, we apply our results to the problem of absolute factorization, that is factorization over $\bar{\kp}$. Note that rational factorization can be seen as a subroutine of absolute factorization. Given $F\in \kp[x,y]$ separable with respect to $y$, we represent the absolute factorization of $F$ as a family of pairs
$$
\{(P_1,q_1),\ldots,(P_t,q_t)\}
$$
where $q_j\in \kp[z]$ is separable, $P_j\in \kp[x,y,z]$ satisfies $\deg_z P_j < \deg q_j$, the bidegree of $P_j(x,y,\phi)$ is constant when $\phi$ runs over the roots of $q_j$  and 
$$
F(x,y)=\prod_{j=1}^t \prod_{q_j(\phi)=0} P_j(x,y,\phi) \in \bar{\kp}[x,y]
$$
is the irreducible factorization of $F$ in $\bar{\kp}[x,y]$. This representation is not unique. We have that $t$ is smaller or equal to the number $r$ of irreducible rational factors, with equality if and only if the $q_j$'s are irreducible.  In analogy to the rational case, we denote by $\bar{r}=\sum \deg q_j$ the number of irreducible absolute factors of $F$. 
We represent the absolute analytic factorization of $F$ in $\bar{\kp}[[x]][y]$ exactly in the same way, the ring $\kp[x]$ being replaced by $\kp[[x]]$ (Section \ref{Sabs}). We denote by $\bar{s}$ the number of irreducible analytic absolute factors of $F$, and we introduce $\bar{\C}(n)$ for the complexity of computing the $n$-truncated absolute analytic factorization of $F$.

\begin{thm}\label{abs1}
Suppose that $p=0$ or $p>d_x(2d_y-1)$ and let $m:=\max(q,d_x+1)$. There exists a deterministic algorithm that, given $F\in \kp[x,y]$ satisfying hypothesis $(H)$, returns its absolute factorization with at most
$$
\mathcal{O}(md_y\bar{s}^{\omega-1}+\bar{r}d_xd_y^2)+\bar{\C}(m)\subset \sO(d_xd_y^{\omega+1})+\bar{\C}(d_x d_y)
$$
arithmetic operations over $\kp$. We can take $m=d_x$ if $F$ is locally absolutely irreducible along the fiber $x=0$. 
\end{thm}

This result has to be compared to \citep{CL}, Proposition 12, where the authors get complexity $\mathcal{O}(d^{\omega+1}+\bar{r}d_x d_y^2)$ for absolute factorization, where $d$ is the total degree of $F$. Note that in contrast to Theorem \ref{t2}, we assume here that $\kp$ has cardinality greater or equal to $d_x(2d_y-1)$. If we only pay attention to the number of absolutely irreducible factors, we can avoid this hypothesis and we can deal with the only $(d_x+1)$-truncation order.
\vskip2mm
\noindent

\begin{thm}\label{abs2}
There exists a deterministic algorithm that, given $F\in \kp[x,y]$ satisfying hypothesis $(H)$, returns the number $\bar{r}$ of irreducible absolute factors of $F$ with at most 

- $\mathcal{O}(d_xd_y\bar{s}^{\omega-1})+\C(d_x)$ operations over $\kp$ if $p=0$  or $p>d_x(2d_y-1)$.

- $\mathcal{O}(kd_xd_y\bar{s}^{\omega-1})+\sO(k)\C(d_x)$ operations over $\fp_p$ if $\kp=\fp_{p^k}$.

\end{thm}

This result has to be compared to \citep{CL}, Proposition 12, where the authors get complexity $\mathcal{O}(d^{\omega+1})$ for computing the number of irreducible absolute factors. As mentionned already, the great advantage of our algorithm is that, when working over a regular fiber, the number of absolute analytic factors to recombine is always $d_y$, while working over critical fibers reduces this number to $\bar{s}\le d_y$. More precisely, if $e_i$ and $f_i$ stand respectively for the ramification indices and residue degrees of the $s$ rational places of $C$ over $\kp$, the difference beewteen $\bar{s}$ and $d_y$ is measured by the formulas 
$$
\bar{s}=\sum_{i=1}^r f_i \quad \le \quad d_y=\sum_{i=1}^r e_i f_i.
$$
Hence, the more ramified the fiber is, the more we gain during the recombination step. Of course, in counterpart, we have to perform analytic factorization along a critical fiber. 

\paragraph*{Example.}
Here is a very simple illustrating example. Suppose for instance that 
$$
F=(y^{a}+\sqrt 2 x^b +x^b y^a)(y^{a}-\sqrt 2 x^b +x^b y^a)
$$
for some coprime integers $a,b$. Then, the curve $F=0$ has only one rational place over $x=0$, with ramification index $e=a$ and residual degree $f=2$. Moreover, $F$ is non degenerate with respect to its Newton polytope. It follows in particular that $m=d_x+1$. After some monomial change of coordinates, we can apply an absolute Hensel lifting strategy which leads to $\bar{\mathcal{C}}(d_x)\subset \sO(d_x d_y)$.
Since both $\bar{s}$ and $\bar{r}$ are constant, it follows from Theorem \ref{abs1} and \ref{abs2} that we compute the number of absolute factors and the absolute factorization of $F$ with respective complexities $\sO(d_xd_y)$ and $\sO(d_xd_y^2)$, which have to be compared to the complexities $\sO(d_x d_y^{\omega})$ inherent to the choice of a regular fiber \citep{CL}. Of course, this is a very special example. In general, it would be really interesting to know both in practice and in theory when one approach is better than an other.

\paragraph*{Main line of the proofs.} 
The approach we propose to solve the problem of recombinations of analytic factors follows closely that of Lecerf \citep{L}. Namely, we use logarithmic derivatives in order to reduce a multiplicative recombination problem to an additive recombination problem. Then, a simple observation shows that we need to test if some rational function $G/F$ has all its residues $\rho_k$'s in the subfield $\bar{\kp}\subset \overline{\kp(x)}$, a problem that can be reduced to a divisibility test by $F$ for zero or big enough characteristic. Note that in contrast to \citep{L}, we do not make any assumption on the cardinality of the field so that we need to take care to the case when the leading coefficient of $F$ is not invertible in $\kp[[x]]$. For small positive characteristic, we test $x$-independance of the residues thanks to an $\fp_p$-linear operator introduced by Niederreiter for univariate factorization \citep{Nied} and extended to the bivariate case by Lecerf \citep{L}. Hence, linear algebra over $\fp_p$ appears, explaining that our complexity results are expressed only for finite fields when the characteristic is small.  When the fiber is regular, residues in $\bar{\kp}$ turn out to be  a sufficient condition for solving recombinations. Along a critical  fiber, this is not the case anymore. The basic idea is to introduce extra linear equations that depend on the higher truncated analytic factors. To this aim, we introduce the \textit{separability order} of $F$ which in the monic case,  coincides with the maximal $x$-valuation of $\partial_y F(x,\phi)$ when $\phi$ runs over the roots of $F$. We show that this integer gives an upper bound for the required precision. If we know moreover that the residues of $G/F$ lie in the subfield $\kp\subset \bar{\kp}$, we show that we can improve this upper bound. This is the kind of arguments that allows us to prove $(2)$ in Theorem \ref{t1}. Finally, we extend our results to the absolute case by using a Vandermonde matrix that allows to reduce  $\bar{\kp}$-linear algebra to $\kp$-linear algebra.

\paragraph*{Organization of the paper.}
In Section \ref{S2}, we introduce our main notations and we explain the recombination problem. In Section \ref{S3}, we solve the recombination problem along a critical fiber. In Section \ref{S4}, we pay attention to the subproblem of counting the number of factors and we give in particular an irreducibility test. We discuss moreover some combinatorial approaches for solving recombinations of some so-called reasonnably ramified polynomials. In Section \ref{S5}, we give explicit equations for constant residues, mainly following \citep{L}. In Section \ref{S6}, we develop the algorithms underlying Theorem \ref{t2} and \ref{t1} and we study their complexities. We consider the case of locally irreducible polynomials and we prove Theorem \ref{tirr} in Subsection \ref{sslocirr}. In Section \ref{Sabs}, we pay attention to absolute factorization and we prove theorems \ref{abs1} and \ref{abs2}. In Section \ref{Spolytope}, we consider the case of non degenerate polynomials with respect to their $P$-adic Newton polytopes. We conclude in Section \ref{S8}.

\section{Factorization, recombinations, residues.}\label{S2}

We explain here the strategy developed by Lecerf in \citep{L} for solving recombinations in the regular case, and we show that some problems occur when working along a critical fiber. For convenience to the reader, we tried to follow the notations of \citep{L}. In all of this section, we suppose that $F$ is primitive with respect to $y$, a situation that can be reached with a negligeable cost for our purpose. For convenience, we only pay attention to rational factorization, the absolute case being treated separately in Section \ref{Sabs}.

\subsection{The recombination problem}\label{ss2.1}

We normalize $F$ by requiring that its leading coefficient with respect to $y$ has its first non zero coefficient equal to $1$. The polynomial $F$ thus admits a unique rational factorization
\begin{equation}\label{factorat}
F=F_1\cdots F_r \in \kp[x,y],
\end{equation}
where each $F_j\in \kp[x,y]$ is irreducible, with leading coefficient with first non zero coefficient equal to $1$. On another hand, $F$ admits a unique analytic factorization of the form
\begin{equation}\label{factoan}
F=u\F_1\cdots \F_s \in \kp[[x]][y]
\end{equation}
where the $\F_i\in \kp[[x]][y]$ are irreducible with leading coefficient $x^{n_i}$, $n_i\in \np$ and $u\in \kp[x]$, $u(0)\ne 0$. Hence, each rational factor $F_j$ has a unique normalized factorization
\begin{equation}\label{prod}
F_j=c_j \F_1^{v_{j1}}\cdots \F_r^{v_{jr}}, \,\quad j=1,\ldots,r.
\end{equation}
for some polynomial $c_j\in \kp[x]$, $c_j(0)=1$.
The recombination problem consists to compute the exponent vectors 
$$
v_j=(v_{j1},\ldots,v_{jr})\in \np^{r}
$$
for all $j=1,\ldots,r$. Then, the computation of the $F_j$'s follows easily. Since $F$ is squarefree by hypothesis, the vectors $v_j$ form a partition of $(1,\ldots,1)$ of length $r$. In particular, they form up to reordering the reduced echelon basis of the vector subspace they generate over any given field $\fp$. In positive characteristic, our algorithm will have to solve linear equations both over $\kp$ and over $\fp_{p}$. Hence, in order to unify our notations, we consider for a while the recombination problem over $\fp$ a fixed given subfield of $\kp$. Namely, we want to compute a basis of the following $\fp$-vector space
$$
S:=\left\langle v_1,\ldots,v_r\right\rangle_{\fp}\subset \fp^s.
$$
Hence, solving recombinations essentially reduces to find a system of $\fp$-linear equations that determine $S\subset \fp^s$. If not specified, all vector spaces we introduce from now are defined over $\fp$, keeping in mind that $\fp$ will have to play the role of $\kp$ or $\fp_p$.

\paragraph*{Truncated functions.} Given $\G \in \kp[[x]][y]$, we denote by
$
[\G]^n \in \kp[x,y]
$
the canonical representant of $\G$ modulo $(x^{n})$. We call it the $n$-truncation of $\G$. We will use also the notation
$$
[\G]_n:=\G - [\G]^n\qquad {\rm and }\qquad [\G]_n^m:=[\G]^m-[\G]^n
$$
for lower truncation of functions, with convention that $[\G]_n^m=0$ for $n\ge m$. In other words, we put to zero all coefficients of monomials with $x$-degree $< n$.

\subsection{Recombination and residues.}\label{ss2.2}
The key point to solve recombinations is to reduce a multiplicative problem to a linear algebra problem thanks to the logarithmic derivative operator. Let $\hat{\F_i}$ stands for the quotient of $F$ by $\F_i$. Let $\mu=(\mu_1,\ldots,\mu_s)\in \fp^s$. Applying logarithmic derivative with respect to $y$ to (\ref{prod}) and multiplying by $F$ we get the key characterization
\begin{equation}\label{maineq}
\mu\in S \iff \exists \,\alpha_1,\ldots,\alpha_r\in \fp\quad |\quad \sum_{i=1}^s \mu_i \hat{\F_i}\partial_y\F_i=\sum_{j=1}^r \alpha_j \hat{F_j}\partial_y F_j.
\end{equation}
The reverse implication holds thanks to the separability assumption on $F$ (\citep{L}, Lemma 1). The key idea is to derive from (\ref{maineq}) a system of linear equations for $S$ that depends only on the $(d_x+1)$-truncated polynomial
$$
G_{\mu}:=\sum_{i=1}^s \mu_i \big[\hat{\F_i}\partial_y\F_i\big]^{d_x+1}\,\, \in \,\,\kp[x,y].
$$
Let $G=G_{\mu}$ and let us denote by $\rho_k=\rho_k(\mu)$ the residues
$$
\qquad \qquad \rho_{k}:=\frac{G(x,y_k)}{\partial_y F(x,y_k)}\,\, \in \,\,\overline{\kp(x)},\qquad k=1,\ldots,d_y
$$
of $G/F$ at the roots $y_k\in \overline{\kp(x)}$ of $F$. These residues are well defined thanks to the separability assumption on $F$. 
We get from (\ref{maineq}) that 
\begin{equation}\label{residue}
\quad\mu\in S \,\,\Longrightarrow \,\, \rho_k\in \fp\quad\forall\, k=1,\ldots, d_y.
\end{equation}
In particular, we have an inclusion of $\fp$-vector spaces
$$
S\subset V(\lp):=\Big\{\mu\in \fp^s\,|\,\, \rho_k \in \lp,\,\, k=1,\ldots,d_y\Big\}
$$
for any subfield $\lp\subset \overline{{\kp}(x)}$. In the regular case, the reverse inclusion holds as soon as $\lp\subset \bar{\kp}$, thanks to the following proposition (\citep{L}, Lemma 2).
\begin{proposition}\label{regular}
Suppose that $F(0,y)$ is separable of degree $d_y$. Then $S=V(\bar{\kp})$.
\end{proposition}
In characteristic zero or high enough, we get that  $\mu\in S$ if and only if $\rho'_k=0$ for all $k$, a condition that can be traduced into a finite number of linear equations over $\kp$. In small positive characteristic $p$, we have that $\rho_k'=0$ implies that $\rho_k\in \overline{{\kp}(x^p)}$ and we use then the Niederreiter operator in order to get some extra $\fp_{p}$-linear equations that allow to test $\rho_k\in \bar{\kp}$ (see Section 3). 

\vskip2mm
\noindent

Unfortunately, the equality $S=V(\bar{\kp})$ in Proposition \ref{regular} no longer holds along a critical fiber, as illustrated by the following example.

\begin{example}\label{ex1}
Let $F=y^6-(y-x)^2\in \qp[x,y]$. We see that $F(0,y)$ has a double root so that the fiber $x=0$ is critical. We compute that $F$ has $s=5$ irreducible analytic factors over $\qp$ and $r=2$ rational factors.
Two of the analytic factors of $F$ have $x$-adic expansions
$$\F_1=y-x-x^3+\cdots,\qquad \F_2=y-x+x^3+\cdots$$
Since $d_x=2$, it follows that 
$$
[\hat{\F_1}\partial_y\F_1]^{d_x+1}=[\hat{\F_2}\partial_y\F_2]^{d_x+1}.
$$ 
In particular, the vector $\mu=(1,-1,0,0,0)\in \fp^5$ gives the null polynomial $G_{\mu}=0$, hence the trivial relation $\mu\in V(\bar{\qp})$. On another hand, we can check that $\mu\notin S$. Hence,  Proposition \ref{regular} doesn't hold in that case. 
\end{example}

\vskip2mm
\noindent
This example suggests to quotient $V(\bar{\kp})$ by the vector subspace $Z$ of relations $G_{\mu}=0$.
Unfortunately, we could not prove that the isomorphism $S\simeq V(\bar{\kp})/Z$ always hold, although this is the kind of approach we will follow in order to compute the number of irreducible factors (Subsection \ref{ssNumFac}). Moreover, even if such an isomorphism holds, it does not allow in general to compute the reduced  echelon basis of $S$ thanks to linear algebra (see Subsection \ref{ss4.4}).  Hence, we rather privilegiate to reduce recombinations to linear algebra. To do so, we need  to introduce extra equations for $S$. Not surprisingly, these equations will depend  now on the analytic factors of $F$ truncated up to some higher precision, this precision being closely related to the valuation of the discriminant.

\section{Recombinations along a critical fiber}\label{S3}

In Subsection \ref{ssorder}, we introduce the notion of separability order of $F$. This integer will measure how much the fiber $x=0$ is critical for $F$ and will play the role of an upper bound for the truncation order of the analytic factors. In Subsection \ref{sssolve}, we solve the recombination problem along a critical fiber. 
We keep the same notations and hypothesis as in the previous section. In particular, $F$ is primitive with respect to $y$.

\subsection{The separability order}\label{ssorder}

To each analytic factor $\mathcal{F}_i$ of $F$, we associate the integers 
$$
r_i:=\val_x \Res_y(\F_i,\hat{\F_i}),\qquad \delta_i:=\val_x \Disc_y(\F_i),\qquad d_i:=\deg_y(\F_i).
$$
Here, $\Res_y$ and $\Disc_y$ stands for the usual resultants and discriminants with respect to $y$, and $\val_x$ stands for the $x$-adic valuation of $\kp[[x]]$. 
We introduce the rational number
$$
q_i:=\frac{r_i+\delta_i}{d_i}
$$
and we denote by $N=N(F)$ the integer:
$$
N:=\max\big\{\lfloor q_{1} \rfloor,\ldots,\lfloor q_{r} \rfloor\big\}.
$$
The integer $N$ measures in some sense how critical the fiber $x=0$ is for the curve $F=0$.  We call it the \textit{separability order} of $F$ along the fiber $x=0$. In particular, we have $N=0$ if $F(0,y)$ is separable of degree $d_y$ (the converse is false, take for instance $F=y^2-x$). The integer $N$ will play the role of an upper bound for the truncation order that allows to solve recombinations. The following lemma summarizes its main properties.  We recall that the standard $x$-adic valuation $\val_x$ of the complete field $\kp((x))$ uniquely extends to a valuation on its algebraic closure $\overline{\kp((x))}$, that we still denote by $\val_{x}$. 

\begin{lemma}\label{order}
\begin{enumerate}
\item We have equality 
\begin{equation}
q_1d_1+\cdots +q_sd_s=\val_x \Disc_y(F).
\end{equation}
\item Let $\phi$ be a root of $\F_i$, and denote by $n_i$ the $x$-valuation of the leading coefficient of $\F_i$. We have the relation
\begin{equation}\label{val}
q_i=\val_x\partial_y F(\phi)+\frac{(d_y-2)n_i}{d_i}.
\end{equation}
In particular, if the leading coefficient of $F$ is invertible in $\kp[[x]]$, we have that
$$
\{q_1,\ldots,q_s\}=\big\{\val_x \partial_y F(\phi),\,\,\phi\,\,{\rm roots\,\,of}\,\, F\}.
$$
\end{enumerate}
\end{lemma}

\begin{proof}
By the multiplicative properties of the discriminant and the resultant, we get that
\begin{eqnarray*}
\qquad\qquad\Disc_y (F)&=&\prod_{i=1}^s \Disc_y(\F_i) \prod_{1\le i<j\le s} \Res_y(\F_i,\F_j)^2\\
&=&\prod_{i=1}^s \Disc_y(\F_i)\prod_{i=1}^s \Res_y(\F_i,\hat{\F_i}),
\end{eqnarray*}
and $(1)$ follows directly by applying $\val_x$ to this equality.
Let now $\phi$ be a root of $\F_i$. We thus have
$$
\partial_y F(\phi)=\hat{\F_i}\partial_y \F_i(\phi).
$$
On another hand, we have the product formula
$$
\prod_{\F_i(\phi)=0}\hat{\F_i} \partial_y \F_i(\phi)=\Res_y(\F_i,\hat{\F_i} \partial_y \F_i)\lc(\F_i)^{1-d_y},
$$
where $\lc(\F_i)$ stands for the leading coefficient of $\F_i$ and where the left hand side product runs over all roots of $\F_i$. Combined with the multplicative property of the resultant 
$$
\Res_y(\F_i,\hat{\F_i} \partial_y \F_i)=\Res_y(\F_i,\hat{\F_i})\Res_y(\F_i,\partial_y\F_i)
$$
and with its relation to the discriminant
$$  
\Res_y(\F_i,\partial_y \F_i)=\lc(\F_i) \Disc_y(\F_i),
$$
we get the formula
$$
\prod_{\F_i(\phi)=0}\hat{\F_i} \partial_y \F_i(\phi)=\lc(\F_i)^{2-d_y}\Res_y(\F_i,\hat{\F_i})\Disc_y(\F_i).
$$
Since $\val_x$ is invariant under the $\kp((x))$-automorphisms of the algebraic closure of $\kp((x))$, point $(2)$ follows by applying $\val_x$ to the previous equality and by dividing by the degree $d_i$ of $\F_i$.
\end{proof}

\begin{remark}\label{remtrunc}
Lemma \ref{order} implies in particular that 
$$
N\le \frac{\val_x \Disc_y(F)}{d}\le \frac{d_x(2d_y-1)}{d},
$$
where 
$d:=\min\{d_i,\,i=1,\ldots,s\}$ stands for the minimal degree of the $\mathcal{F}_i$'s.
In particular, $N\in \mathcal{O}(d_xd_y)$. The following generalization of Example \ref{ex1}, suggested to us by Eduardo Casas-Alvero, shows that $N$ may reach this order of magnitude.
\end{remark}

\begin{example}\label{ex2}
Let $F(x,y):=(y-x^m)^2+y^n\in \qp[x,y]$, with $n\ge 3$ odd. Then $(0,0)$ is the unique point of the curve $F=0$ that is ramified over $x=0$. We can show that $F$ admits a unique irreducible analytic factor $\F_1$ vanishing at $(0,0)$, with degree $d_1=2$. It follows that
$$
\delta_1=\val_x \Disc_y(F)\quad {\rm and} \quad r_1=0,
$$
while $\delta_i=r_i=0$ for all $i> 1$. We compute here that $\val_x \Disc_y(F)=mn$. It follows that
$$N=\Big\lfloor\frac{r_1+\delta_1}{d_1}\Big\rfloor=\Big\lfloor\frac{mn}{2}\Big\rfloor=\Big\lfloor\frac{d_x d_y}{4}\Big\rfloor,$$ 
which is of the order of magnitude of $d_xd_y$.
\end{example}

\subsection{Solving recombinations along a critical fiber.}\label{sssolve}

We can derive from (\ref{maineq}) an other obvious source of equations for $S$. Namely, let us introduce for $n\in \np$ the $\fp$-vector subspace
$$
W^n:=\Big\{\mu\in \fp^s\,\,|\,\,\, \sum_{i=1}^s \mu_i \big[\hat{\F_i}\partial_y\F_i\big]^n_{d_x+1}=0\Big\},
$$
with convention $W^n=\fp^s$ when $n\le d_x+1$. For a question of degree, (\ref{maineq}) implies that we have the inclusions
$$
S\subset W^n\quad \forall n \in \np.
$$
Our next result ensures that the separability order gives an {\it a priori} upper bound for $n$ for which $W^n$ provides enough extra equations to solve the recombination problem. 

\begin{theorem}\label{critic} 
We have $S=V(\bar{\kp})\cap W^{n}$ for all $n>N$. 
\end{theorem}

In particular, if $N\le d_x$, then the recombinations are solved by the same system of linear equations as in the regular case:

\begin{corollary}\label{cor1}
Suppose that $N\le d_x$. Then $S=V(\bar{\kp})$.
\end{corollary}

\begin{remark}
This corollary implies in particular that all polynomials that are non degenerate with respect to their Newton polytope satisfy $V(\bar{\kp})=S$ (see Section \ref{Spolytope}). 
\end{remark}

\begin{remark}\label{remfiber}
Suppose that $F$ is separable with respect to $y$. For $\alpha\in \mathbb{P}^1_{\kp}$, let us denote by $N_{\alpha}$ the separability order of $F$ over the fiber $x=\alpha$. From inequalities, 
$$
\sum_{\alpha\in \mathbb{P}^1_{\kp}} N_{\alpha}\le \sum_{\alpha\in \mathbb{P}^1_{\kp}} \val_{x-\alpha} \Disc_y(F) =\deg_x \Disc_y(F) \le d_x(2d_y-1),
$$
we deduce that there always exist a fiber for which $N_{\alpha}\le d_x$ as soon as $\kp$ has cardinality $\ge 2d_y-3$.  
\end{remark}

In order to prove Theorem \ref{critic}, we need to prove two preliminary lemmas. The first lemma is a key lemma that will be used many times in the paper.

\begin{lemma}\label{gao}
Let $\kp\subset \lp \subset \overline{\kp((x))}$ be a field and let $G\in \kp[x,y]$, with $\deg_y \, G < d_y$. The residues $\rho_k$ of $G/F$ all lie in $\lp$ if and only if $G$ is $\lp$-linear combination of the $\hat{E_j}\partial_y E_j$'s, where the $E_j$'s stand for the irreducible factors of $F$ over $\lp$.  
\end{lemma}

\begin{proof}
One direction is clear: if $G=\sum \alpha_j\hat{E_j}\partial_y E_j$, then $\rho_k=\alpha_j\in \lp$ where $j$ is determined by condition $E_j(x,y_k)=0$. Suppose now that $\rho_k\in \lp$. Thanks to the degree assumption on $G$ and the separability assumption on $F$, we have the partial fraction decomposition
$$
\frac{G}{F}=\sum_{k=1}^{d_y} \frac{\rho_k}{y-y_k}.
$$ 
Let $\tau\in \Gamma:=\Aut\Big(\overline{\lp(x)}/\lp(x)\Big)$ acts on this equality. By assumption, $\tau$ leaves both $G/F$ and $\rho_k$ fixed. Hence, we get 
$$
\sum_{k=1}^{d_y} \frac{\rho_k}{y-\tau(y_k)}=\sum_{k=1}^{d_y} \frac{\rho_k}{y-y_k}=\sum_{k=1}^{d_y} \frac{\rho_{k_{\tau}}}{y-\tau(y_k)},
$$
the second equality using that $\tau$ permutes the roots of $F$. Here, the notation $k_{\tau}$ stands for the unique index such that $\phi_{k_{\tau}}=\tau(y_k)$. The partial fraction decomposition being unique, previous equality implies that
$$
\rho_k=\rho_{k_{\tau}} \quad \forall\, \tau \in \Gamma.
$$
Since $\Gamma$ acts transitively on the set of roots of each $\lp$-irreducible factor $E_{j}$, it follows that $\rho_k=\rho_{k'}$ whenever $E_j(x,y_k)=E_j(x,y_{k'})=0$. Hence, there exist constants $\alpha_1,\ldots, \alpha_{\ell}\in \lp$ such that
$$
\frac{G}{F}=\sum_{j=1}^{\ell} \alpha_j \Big(\sum_{k|E_j(y_k)=0} \frac{1}{y-y_k}\Big)=\sum_{j=1}^{\ell} \alpha_j \frac{\partial_y E_j}{E_j}.
$$
The result follows from multiplication by $F$.
\end{proof}

The next lemma computes the valuations of the roots of $F$. 

\begin{lemma}\label{lval}
Let $\F\in \kp[[x]][y]$ be an irreducible polynomial of degree $d$ in $y$. Let $a$ and $b$ stand respectively for the valuation of the leading coefficient and the constant coefficient of $\F$ seen as a polynomial in $y$. Let $\phi\in \overline{\kp((x))}$ be a root of $\F$. Then $\val_x (\phi)=(b-a)/d$ and either $a$ or $b$ is equal to $0$.
\end{lemma}

\begin{proof}
Since $\F$ is irreducible, at least one of its coefficient has valuation $0$. Hence, if both $a$ and $b$ are non zero, then its Newton polytope would contain at least two distinct compact edges (Section \ref{Spolytope}). This is impossible since $\F$ is irreducible. Let $\N$ stands for the norm of the field extension of $\kp((x))$ defined by $\F$. Then $\N(\phi)$ is equal to the quotient of the constant coefficient of $\F$ by its leading coefficient. Hence $\val_x \N(\phi)=b-a$ and we conclude thanks to the relation $\val_x \phi= \val_x\N(\phi)/\deg(\phi)$.
\end{proof}

\paragraph*{Proof of Theorem \ref{critic}.} We already saw that $S\subset V(\bar{\kp})\cap W^n$ and we need to prove the reverse inclusion when $n> N$. Let $\mu\in V(\bar{\kp})\cap W^n$. Thanks to the previous lemma, and by definition of $W^n$, we deduce that there exists some constants $\alpha_j\in \bar{\kp}$ such that
\begin{equation}\label{eq4}
\sum_{i=1}^s\mu_i \big[\hat{\F_i} \partial_y \F_i\big]^m =\sum_{j=1}^{\ell}\alpha_j \hat{E_j} \partial_y E_j,
\end{equation}
where $m=\max(d_x+1,n)$ and where the $E_j$'s stand for the irreducible factors of $F$ over $\bar{\kp}$.  Let $\phi\in \overline{\kp((x))}$ be a root of $\F_i$ and let $j$ be the unique index such that $E_j(\phi)=0$. Using the relations
$$
\hat{\F_i}\partial_y\F_i(\phi)=\hat{E_j} \partial_y E_j(\phi)=\partial_y F(\phi),
$$
we get by evaluating $(\ref{eq4})$ at $\phi$ an equality
\begin{equation}\label{eq5}
(\mu_i-\alpha_{j})\partial_yF(\phi)= x^{m}R(\phi)
\end{equation}
for some $R\in \kp[[x]][y]$. We need a lower bound on the valuation of $R(\phi)$. We remark that the coefficient of $y^{d_y-1}$ in $\partial_y\F$ is equal to $d_y\lc_y(F)$. Since the leading coefficient of $F$ is a polynomial in $x$ of degree at most $d_x$, equation (\ref{eq5}) implies that $R$ has $y$-degree $\le d_y-2$. Hence, ultrametric inequality combined with Lemma \ref{lval} gives
$$
\val_x R(\phi)\ge \min\{\val_x \phi^i,\,\,i=0,\ldots,d_y-2\}\ge -\frac{(d_y-2)n_i}{d_i},
$$
(recall that  $x^{n_i}$ stands for the leading coefficient of $\F_i$). Suppose that $\mu_i\ne \alpha_j$. Hence, (\ref{eq5}) gives 
\begin{equation*}
\val_{x} \partial_y F(\phi)\ge m-\frac{(d_y-2)n_i}{d_i}.
\end{equation*}
By Lemma \ref{order}, this is equivalent to that $m\le q_i$, contradicting our hypothesis $m=\max(d_x+1,n)> N$. It follows that $\mu_i=\alpha_{j}$. Combined with (\ref{eq4}), we get that
$$
G:=\Big[\sum_{i=1}^s\mu_i \hat{\F_i} \partial_y \F_i\Big]^{d_x+1}=\sum_{i=1}^s\mu_i \hat{\F_i} \partial_y \F_i.
$$
In particular, the residues of $G/F$ all lie in $\fp\subset\kp$, and it follows from Lemma \ref{gao} that  
$$
G=\sum_{i=s}^r\mu_i \hat{\F_i} \partial_y \F_i=\sum_{j=1}^r c_j \hat{F_j} \partial_y F_j.
$$
for some $c_j\in \kp$. Since $\F_i$ is coprime to $\hat{\F_i}\partial_y\F_i$ by hypothesis, this relation forces equality $\mu_i=c_j$ when $\F_i$ divides $F_j$. It follows that $\mu\in S$. $\hfill\square$

\section{Counting the number of irreducible factors}\label{S4}

We show here how to bound the number of factors with the $d_x+1$-truncation order and we deduce a deterministic  irreducibility test that requires the only $2d_x$-truncation order. We still suppose that $F$ is primitive with respect to $y$.

\vskip2mm
\noindent
\subsection{An upper bound for the number of factors}\label{ssNumFac}

Example \ref{ex1} suggests to introduce the vector subspace $Z$ of vectors $\mu$ whose associated trunacted polynomial $G_{\mu}$ is null, that is
$$
Z:=\Big\{\mu\in \fp^s\,\,|\,\, \sum_{i=1}^s\mu_i \big[\hat{\F_i} \partial_y \F_i\big]^{d_x+1}=0\Big\}.
$$
We have the following result.

\begin{proposition}\label{NumFac} 
We have $V(\fp)=S\oplus Z$.
\end{proposition}

\begin{proof}
We already saw that $S\subset V(\fp)$, while the inclusion $Z\subset V(\fp)$ trivially holds. Hence, we get an inclusion $S+Z\subset V(\fp)$. Let us show the reverse inclusion. If $\mu \in V(\fp)$, it follows from Lemma \ref{gao} that $G_{\mu}$ is $\fp$-linear combinations of the irreducible factors of $F$ over $\kp$. It follows that
$$
\sum_{i=1}^s \mu_i [\hat{\F_i} \partial_y\F_i]^{d_x+1}= \sum_{i=1}^s \alpha_i [\hat{\F_i} \partial_y\F_i]^{d_x+1},
$$
for some $\alpha=(\alpha_1,\ldots,\alpha_s)\in S$. In particular, $\mu-\alpha\in Z$. Equality $V(\fp)=S+Z$ follows. Finally, if $\mu\in S\cap Z$ we get that
$$
\sum_{i=1}^s \mu_i \hat{\F_i} \partial_y\F_i=\sum_{i=1}^s \mu_i \big[\hat{\F_i} \partial_y\F_i]^{d_x+1}=0,
$$
so that $\mu=0$ by linear independance of the $\hat{\F_i} \partial_y\F_i$'s. It follows that $V(\fp)=S\oplus W$. 
\end{proof}

\begin{remark}\label{remZ}
We have $Z=0$ as soon as the separability order satisfies $N\le d_x+1$. Namely, we have in that case $S=V(\bar{\kp})$ by Theorem \ref{critic} and we conclude thanks to the inclusion $S\oplus Z= V(\kp)\subset V(\bar{\kp})$. 
\end{remark}

For fields of positive characteristic, we can take $\fp$ as the prime field of $\kp$, in which case the Niederreister operator leads to an explicit system equations for $V(\fp)$ (see Section \ref{S5}). Hence, Proposition \ref{NumFac} allows to compute the number of irreducible factors 
$$
r=\dim_{\fp} \, V(\fp) -\dim_{\fp} \, Z
$$
with linear algebra from the $(d_x+1)$-truncated analytic factors only. For fields of characteristic zero, testing whether the residues lie in $\kp$ is a much harder task. In that case, we only get equations for $V(\bar{\kp})$, so that Proposition \ref{NumFac} {\it a priori} allows only to compute the upper bound 
$$
r\le \dim_{\fp} V(\bar{\kp})-\dim_{\fp} Z.
$$
This problem motivates to explore in more details the relations beetween $V(\fp)$ and $V(\bar{\kp})$.

\subsection{On the relations beetween $V(\fp)$ and $V(\bar{\kp})$}\label{ss3.3}

In regards to the Proposition \ref{NumFac}, we may ask whether equality $V(\fp)=V(\bar{\kp})$ holds. We could not prove nor disprove this equality. However, we give here some conditions under which it holds. Let us first note the following lemma.

\begin{lemma}\label{KKbar}
If all the absolute factors of $F$ are defined over $\kp$, then $V(\kp)=V(\bar{\kp})$.
\end{lemma}

\begin{proof}
By Lemma \ref{gao}, if $\mu\in V(\bar{\kp})$, then
$G_{\mu}=\sum \alpha_j \hat{E}_j \partial_y E_j$ for some $\alpha_j\in \bar{\kp}$, and where the $E_j$'s stand for the irreducible absolute factors of $F$. By assumption, we have that $E_j\in \kp[x]$. Applying $\tau\in \Aut_{\kp}(\bar{\kp})$ to the previous equality, and using that $G_{\mu}$ has coefficients in $\kp$, we get that 
$$
\sum_j \tau(\alpha_j)\hat{E}_j \partial_y E_j=\sum_j \alpha_j\hat{E}_j \partial_y E_j,
$$
which implies that $\tau(\alpha_j)=\alpha_j$ by $\bar{\kp}$-linear independance of the $\hat{E}_j \partial_y E_j$'s. This being true for all $\tau$, it follows that $\alpha_j\in \kp$. Hence $\mu\in V(\kp)$ by Lemma \ref{gao}.
\end{proof}

To each rational factor $F_j$ of $F$, we associate the integer
$$
M_j:=\min \Big\{\lfloor q_i\rfloor,\,\,\F_i \,\,{\rm divides}\,\, F_j\,\,{\rm in }\,\,\kp[[x]][y]\Big\}.
$$
Roughly speaking, $M_j$ measures the minimal contact order of the curve $F_j=0$ with the complementary curve $\hat{F_j}=0$ along the line $x=0$. We denote by 
$$
M:=\max\{M_j,\,\,j=1,\ldots,r\}.
$$
Note the obvious relation $M\le N$ with the separability order.

\begin{proposition}\label{prores}
We have $V(\bar{\kp})\cap W^n=V(\fp)\cap W^n$ for all $n>M$. 
\end{proposition}

\begin{proof} The proof is similar to that of Theorem \ref{critic}. Let $n>M$ and denote by $m:=\max\{n,d_x+1\}$. By Lemma \ref{gao} and by definition of $W^n$ we have $\mu  \in V(\bar{\kp})\cap W^n$ if and only if
\begin{equation}\label{eqphi}
\sum_{i=1}^s \mu_i [\hat{\F_i} \partial_y\F_i]^{m}=\sum_{k=1}^t \alpha_{k} \hat{E}_{k}\partial_y E_{k},
\end{equation}
where the $E_{k}\in\bar{\kp}[x,y]$ stand for the absolutely irreducible factors of $F$. Let us fix $F_j$ a rational factor of $F$. By assumption, there exists $\F_i$ a divisor of $F_j$ such that $q_i\le n$. Let $E_{k}$ be a divisor of $F_j$. Then $E_k$ shares at least one root $\phi\in\overline{\kp(x)}$ with $\F_i$. Hence, by evaluating (\ref{eqphi}) at $\phi$ we get that
$$
(\mu_i-\alpha_k)\partial_y F(\phi)=x^m R(\phi)
$$
for some $R\in \kp[[x]][y]$. Taking $x$-valuations, and reasonning as in the proof of Theorem \ref{critic}, we get that $\mu_i\ne \alpha_k$ implies $q_i\ge m$, a contradiction. Hence $\alpha_k=\mu_i\in \fp$  for all irreducible factors $E_k$ of $F_j$. Repeating this reasonning for all factors $F_j$ of $F$, we deduce by regrouping the factors $E_k$ by conjugacy classes that we have
$$
\sum_{i=1}^s \mu_i [\hat{\F_i} \partial_y\F_i]^{m}=\sum_{j=1}^s c_j \hat{F}_j\partial_y F_j,
$$
for some $c_j$'s in $\fp$. For a degree reason, this is equivalent to that 
$$
\sum_{i=1}^s \mu_i [\hat{\F_i} \partial_y\F_i]^{d_x+1}=\sum_{j=1}^s c_j \hat{F}_j\partial_y F_j\quad {\rm and}\quad \sum_{i=1}^s \mu_i [\hat{\F_i} \partial_y\F_i]_{d_x+1}^m=0
$$
The first equation is equivalent to that $\mu\in V(\fp)$ by Lemma \ref{gao}, while second equation is equivalent to that $\mu\in W^n$ by definition.  
\end{proof}

Last Proposition says in particular that if each irreducible rational factor of $F$ has at least one branch with $q_i\le d_x$, then we have $V(\bar{\kp})=V(\fp)$.  This is the case for instance in Example \ref{ex1}, Section \ref{S2}. Here is a trivial example that illustrates that the converse doesn't hold.

\begin{example}\label{exlin}
Let $F(x,y)=((y-x)^2+y^{10})(y-x)\in \qp[x,y]$. Then $F$ has exactly $3$ anaytic factors over $\qp$ which satisfy
$$
\F_1\equiv (y-x)^2\mod x^4,\quad \F_2\equiv y-x\mod x^4, \quad \F_3=y^8+1+\cdots. 
$$
We find here that $q_1=q_2=10$. In particular, the rational factor $y-x$ of $F$ has a unique branch $\F_2$ and this branch satisfies $q_2>d_x+1=4$. On another hand, we have that
$$
\F_1\equiv \F_2^2\mod x^{d_x+1}\Longrightarrow (1,-2,0)\in Z
$$
and we can show that this is the only possible relation. Hence $\dim Z=1$. Since clearly $\dim S=2$, it follows from Proposition \ref{NumFac} that $\dim V(\qp)=3$. Since $s=3$ is the dimension of the ambient space, it follows that $V(\bar{\qp})=V(\qp)=\qp^3$. Observe that we could not use directly Lemma \ref{KKbar} to show this equality since $F$ has two absolute factors $y-x+iy^5$ and $y-x+iy^5$ that are not defined over $\qp$. 
\end{example}

\subsection{Number of factors. Irreducibility test.}\label{ss3.4}

Proposition \ref{prores} leads to a formula for $r$ that depends only on the $M$-truncated factors:

\begin{corollary}\label{propratres}
The number of rational factors is equal to
$$
r=\dim \,V(\bar{\kp})\cap W^{M+1}-\dim \,Z\cap W^{M+1},
$$
hence can be computed with the only truncated precision $\max(d_x+1,M+1)$.
\end{corollary}

\begin{proof}
We know from Proposition \ref{NumFac} that $V(\fp)=S\oplus Z$. Intersecting with $W^n$, and using that $S\subset W^n$, we get that
$$
S\oplus (Z\cap W^n)=(S\oplus Z)\cap W^n =V(\fp)\cap W^n= V(\bar{\kp})\cap W^n.
$$
for all $n>M$, the last equality thanks to Proposition \ref{prores}. The corollary follows by counting dimensions. 
\end{proof}

Of course, we can not {\it a priori} compute $M$ without knowing the rational factorization so that Corollary \ref{propratres} seems to be useless from a computational point of view. However, it leads to an irreducibility test over $\kp$ with the only $2d_x$-truncated precision.

\begin{corollary}\label{propirr}
The polynomial $F$ is irreducible over $\kp$ if and only if 
$$
\dim\,V(\bar{\kp})\cap W^{2d_x}-\dim\, Z\cap W^{2d_x}=1.
$$
\end{corollary}

\begin{proof}
Suppose that $\dim\,V(\bar{\kp})\cap W^{2d_x}-\dim\, Z\cap W^{2d_x}=1$. Since the inclusion $W^{2d_x}\subset W^n$ holds for all $n\ge 2d_x$ we deduce from Corollary \ref{propratres} that $\dim S=1$. Suppose now that $F$ is irreducible over $\kp$. Then it's enough to show that $M< 2d_x$ by the same argument. 
Suppose on the contrary that $M\ge 2d_x$. Since $F$ is irreducible, we have by definition of $M$ that $q_i\ge 2d_x$ for all $i$. It follows from Lemma \ref{order} that 
$$
\val_x(\Disc_y F)=\sum_{i=1}^s q_i d_i\ge 2d_x \sum_{i=1}^s d_i =2d_xd_y,
$$
which is impossible for a degree reason. 
\end{proof}

\subsection{A combinatorial approach for solving recombinations}\label{ss4.4}

We show here that under some reasonnable conditions, we can compute the factorization of $F$ just by knowing $V(\fp)$ and $Z$, hence from the $(d_x+1)$-truncated analytic factors only. Let us introduce the subset
$$
I:=\Big\{i\in \{1,\ldots,s\},\, \, \, \mu\in Z \Rightarrow \mu_i=0\Big\}.
$$
and let 
$$
L:=\Big\{\mu \in \fp^s,\,\,\, \mu_i=0\,\,\forall \,i\notin I\Big\}.
$$
We denote by $\pi : \fp^s\to L$ the natural projection on $L$. 

\begin{definition} We say that $F$ is reasonnably ramified over $x=0$ if $$\dim\, \pi(V(\fp))=\dim V(\fp) - \dim Z.$$ 
\end{definition}

In other words $F$ is reasonnably ramified if and only if for all $j=1,\ldots,r$, there is an analytic factor $\F_i$ of $F_j$ such that $\mu\in Z$ implies $\mu_i=0$. In particular, if $M\le d_x+1$, then $F$ is reasonnably ramified thanks to the proof of Proposition \ref{prores}. Note that for fields with positive characteristic, we can test if $F$ is reasonnably ramified since we can then compute $V(\fp)$, $Z$ and $L$ (see Subsection \ref{smallchar}). In characteristic zero, this will be the case if we know moreover that $V(\bar{\kp})=V(\fp)$.


\begin{proposition}\label{pcomb}
Suppose that $\kp$ has characteristic zero or strictly greater than $d_y$. Suppose that $F$ is reasonnably ramified over $x=0$. Suppose that $V(\fp)=Z\oplus T$ for some vector subspace $T$ whose reduced echelon basis $(w_1,\ldots,w_r)$ form a partition of $(1,\ldots,1)$. Then, 
$$
F_j=\prim\Big[\lc(F)\prod_{i=1}^s \F_i^{w_{ji}}\Big]^{d_x+1}\quad j=1,\ldots,r
$$
and all rational factors $F_j$  of $F$ can be computed within $\wt{O}(d_x d_y)$ field operations.
\end{proposition}

Here, $\prim$ stands for the primitive part with respect to $y$. In order to prove Proposition \ref{pcomb}, we first need a key lemma. We denote by $R:=\kp[[x]]/(x^{d_x+1})$ and by $f_i$ the class of $\F_i$ in $R$. Since $F$ is not divisible by $x$, we have that $f_i\in R^*[y]$, where $R^*$ stands for the multiplicative group of non zero divisors. In particular, it makes sense to compute $f_i^{k}$ in the total ring of fractions of $R[y]$ for any integer $k\in\zp$. 

\begin{lemma}\label{lprod}
Suppose that $\kp$ has characteristic zero or strictly greater than $d_y$, and let $\mu\in \{0,1\}^s$. Then $\mu\in Z$ if and only if $\prod_{i=1}^r f_i^{\mu_i}=1\in R$. If $\kp$ has characteristic zero, the same conclusion holds with the weaker hypothesis $\mu\in \zp^s$. 
\end{lemma}

\begin{proof}
We have equaliity
$$
\sum \mu_i \hat{f}_i \partial_y f_i=f\frac{\big(\prod f_i^{\mu_i}\big)'}{\prod f_i^{\mu_i}}
$$ 
in the total ring of fractions of $R[y]$. Hence $\mu\in Z$ if and only if $\big(\prod f_i^{\mu_i}\big)'=0$, which is equivalent to that $\prod f_i^{\mu_i}\in R^*(y^p)$, where $p$ stands for the characteristic of $\kp$. If $p>d_y$, and since $\mu_i\in \{0,1\}$, we necessarily have $\prod f_i^{\mu_i}\in R^*$ for a degree reason. If $p=0$ the same holds obviously. In particular, we must have $$\prod f_i^{\mu_i}=\prod f_i^{\mu_i}(\infty):=\prod \lc(f_i)^{\mu_i}=x^k,$$
for some $k\in \zp$. But we know that $\val_x (f_i)=0$ for all $i$, hence we must have $k=0$. 
\end{proof}

\paragraph*{Proof of Proposition \ref{pcomb}.}
we have by assumption that $V(\kp)=S\oplus Z$ where $\pi(Z)=0$ and where the reduced echelon basis $(v_1,\ldots,v_s)$ of $S$ is such that $\pi(v_1),\ldots \pi(v_s)$ are non zero vectors of $\{0,1\}^s$ in reduced echelon form. Hence the same property has to hold for the basis $(w_1,\ldots,w_s)$ of $T$ and up to reordering the $w_j$'s, we must have equality $\pi(w_j)=\pi(v_j)$, forcing relations $w_j-v_j\in Z$. Then the proof of Proposition \ref{pcomb} then follows from Lemma \ref{lprod} combined with relations $(\ref{factoan})$ and $(\ref{prod})$. Since $w_{j}$ has entries in $\{0,1\}$, the complexity for computing $F_j$ belongs to $\tilde{\sO}(d_x \deg_y F_j)$ using fast multiplication in $R[y]$ (see the proof of Proposition \ref{algo1} in Section \ref{S6} for details concerning complexity issues). The last statement then follows by adding this cost over all $j$. $\hfill\square$.

\vskip2mm
\noindent
\begin{example}\label{ex1suite} Let us return to example \ref{ex1} of Subsection \ref{ss2.2}. We have that 
$$
V(\kp)=\big\langle (1,0,0,0,1),(0,1,0,0,1),(0,0,1,1,-1)\big\rangle\quad {\rm and}\quad Z=\langle (1,-1,0,0,0) \rangle.
$$
Hence $\pi(V(\kp))=\big\langle (1,0,0,0,1),(0,1,0,0,0)\big\rangle$ has dimension $\dim V(\kp)-\dim Z=2 (=\dim S)$ and $F$ is reasonnably ramified. Let $T$ be such that $V(\fp)=Z\oplus T$ and such that the reduced echelon basis $(w_1,w_2)$ of $T$ is a partition of $(1,\ldots,1)$. The constraints $w_i\in\{0,1\}^5 \cap V(\kp)$, $w_i\ne 0$ lead to the two possible solutions 
$$
w_1=(1,0,0,0,1)\quad {\rm or} \quad w_1=(1,0,1,1,0).
$$
Using the relation $w_1+w_2=(1,1,1,1,1)$, these solutions correspond respectively to the factorizations
$$
(F_1,F_2)=([\F_1 \F_5]^{d_x+1},[\F_2\F_3\F_4]^{d_x+1})
\quad {\rm or} \quad
(F_1,F_2)=([\F_1 \F_3\F_4]^{d_x+1},[\F_2 \F_5]^{d_x+1})
$$
But we know here that $[\F_1]^{d_x+1}=[\F_2]^{d_x+1}$ since $(1,-1,0,0,0)\in Z$. Hence both solutions determine the irreducible factorization of $F$, as predicted by Proposition \ref{pcomb}.  
\end{example}

\begin{example} Let us return to example \ref{exlin}. We have  
$$
V(\kp)=\langle (1,0,0),(0,1,0),(0,0,1)\rangle \quad {\rm and}\quad Z=\langle (1,-2,0) \rangle.
$$
We have that $\pi(V(\kp))=\langle (0,0,1) \rangle$ has dimension $1<2=\dim V(\kp)-\dim Z$. So $F$ is not reasonnably ramified.  
The family of all possible complementary subspaces $T$ of $Z$ in $V(\kp)$ whose basis form a partition of $(1,1,1)$ is 
$$
T=\langle (1,0,0),(0,1,1)\rangle,\quad T=\langle (0,1,0),(1,0,1)\rangle \quad  {\rm or}\quad T=\langle (0,0,1),(1,1,0)\rangle.
$$
Contrarly to the previous example, only the second solution leads to the good factorization of $F$.
\end{example}

Unfortunately, for more complicated examples, looking for a complementary vector space $T$ of $Z$ in $V(\bar{\kp})$ whose reduced echelon basis form a partition of $(1,\ldots,1)$ might not be an easy task, even though we know such a $T$ exists.

An alternative approach in the zero characteristic case is to use linear algebra over $\zp$. Namely, we can suppose in that case that $\fp=\qp$, so that $V(\fp)\cap \zp^s$ is a free $\zp$-module or rank $\dim V(\fp)$. Recall that the Hermite normal form of a matrix with integer entries is such that the leading entry (first non zero entry) of a nonzero row is positive and strictly to the right of the leading entry of the row above it. Moreover, all entries in a column above a leading entry are nonnegative and strictly smaller than the leading entry. This forces also all entries in a column below a leading entry to be zero. Such a form exists and is unique, and it conserves the row space \citep{sto}. We have:

\begin{proposition}
Suppose that $\kp$ has characteristic zero and that $F$ is reasonnably ramified. Then we can order the set $\{1,\ldots,s\}$ such that $I=\{1,\ldots,\ell\}$ for some $\ell\ge s$. Let $(w_1,\ldots,w_r)$ be the first $r$ vectors of the Hermite normal form of a basis of the free $\zp$-module $V(\fp)\cap \zp^s$.   Then, 
$$
F_j=\prim\Big[\lc(F)\prod_{i=1}^s \F_i^{w_{ji}}\Big]^{d_x+1}\quad j=1,\ldots,r.
$$
If moreover the $w_j$'s have positive entries, then we can compute the $F_j$'s within $\wt{\sO}(d_xd_y)$ field operations. 
\end{proposition}

\begin{proof}
We have by assumption that $V(\kp)=S\oplus Z$ where $\pi(Z)=0$ and where the reduced echelon basis $(v_1,\ldots,v_s)$ of $S$ is such that $\pi(v_1),\ldots ,\pi(v_s)$ are non zero vectors in row echelon form. After reordering the columns as in the Proposition, it follows from the definition of the Hermite normal form that $\pi(w_j)=\pi(v_j)$, which forces $w_j-v_j\in Z\cap \zp^s$. The conclusion then follows from Lemma \ref{lprod}. If the $w_j$'s have positive entries, the computation of all $F_j$'s reduces to a product of polynomials whose total degree sum is $d_y$. This costs $\wt{\sO}(d_x d_y)$. 
\end{proof}

The advantage is that we reduce our recombination problem to linear algebra (over $\zp$): there are efficient algorithms to compute the Hermite normal form over $\zp$ with the same number of arithmetic operations as for the reduced echelon form over $\qp$ \citep{sto}. The difficulty is that some of the $w_j$'s might have negative entries, in which case the complexity of computing the $F_j$'s is more difficult to estimate.

\begin{example}
Let us continue example \ref{ex1suite}. After reordering, we get that the Hermite normal form of the basis of $V(\kp)\cap \zp^4$ is 
$$
\langle w_1=(1,0,1,0,1),w_2=(0,1,0,0,1),(0,0,0,1,-1) \rangle,
$$
leading to the factorization
$$
(F_1,F_2)=([\F_1 \F_3 \F_5]^{d_x+1},[\F_2\F_5]^{d_x+1}).
$$
Here, the vectors $w_1$ and $w_2$ have positive entries, and the computation of the $F_j$'s is fast. Note that the vectors $w_1$ and $w_2$ do not necessarily form a partition of $(1,1,1,1,1)$ anymore. In particular, the analytic factor $\F_4$ has disappeared from the recombination process due to the relation $[\F_4]^{d_x+1}=[\F_5]^{d_x+1}$. 
\end{example}

%
%

\begin{example} Suppose that $s=5$ and that 
$$
S=\langle (1,0,1,0,0), (0,1,0,1,1)\rangle \quad Z=\langle (0,0,1,2,-3).
$$
Then $F$ is reasonnably ramified. the Hermite normal form of the basis of $V(\kp)\cap \zp^5$ is
$$
\big((1,0,0,-2,3),(0,1,0,1,1),(0,0,1,2,-3)\big).
$$
The vector $w_1=(1,0,0,-2,3)$ has now negative entries and the computation of the corresponding factor 
$$
F_1=\frac{f_1 f_5^3}{f_4^2}
$$ is {\it a priori} more expensive. Note that if we had reordered the indices such that $Z=\langle (0,0,2,1,-3)$, we would have obtained $w_1$ and $w_2$ with positive entries. 
\end{example}

\begin{remark}
In general, we can show that if $F$ is reasonnably ramified and if $Z$ is generated by vectors with at most two non zero entries (meaning that all branches with high $q$-invariant intersect at most one other branch), then we necessarily have $w_j \in \np^s$. A concrete example for which it is not the case is given by $F=(y^6-(y-x)^2)(y-x)\in \qp[x,y]$. In that case $Z=\langle (0,0,0,2,-1,-1) \rangle$. 
\end{remark}


\section{Conditions for constant residues. Equations of $V(\lp)$.}\label{S5}

There are two main approaches that allow to determine when the residues of $G/F$ do not depend on $x$. The first approach is related to the first De Rham cohomology group of the complementary set of the affine curve $F=0$ in $\ap^2_{\kp}$. It allows to test whether certain differential meromorphic forms are linear combinations of the logarithmic derivatives of the absolute factors of $F$ by checking closedness. This is the approach followed by Gao \citep{Gao}. 
The second approach, that we will follow here, is based on a divisibility criterion by $F$ and has been developped by Lecerf \citep{L}. Hence, this section is essentially a relecture of Section 1 in \citep{L}, except that we have now to take into account that the leading coefficient of $F$ is not necessarily invertible in $\kp[[x]]$. 

From now on, we fix $\mu\in \fp^s$ and we denote by $G:=G_{\mu}$ the corresponding polynomial and by $\rho_1,\ldots,\rho_{d_y}$ the residues of $G/F$ at the roots of $F$. We denote by $p\ge 0$ the characteristic of $\kp$ and we adopt the convention $\kp(x^p)=\kp$ for $p=0$. 

\subsection{Equations for $V(\bar{\kp})$}\label{ss5.1}
We want to know when the residues $\rho_k\in \overline{\kp(x)}$ lie in the subfield  $\bar{\kp}$. We recall that the usual $\bar{\kp}$-derivation of $\bar{\kp}(x)$ uniquely extends to a $\bar{\kp}$-derivation of $\overline{\kp(x)}$. An obvious necessary condition for that $\rho_k\in \bar{\kp}$ is that the derivatives of the $\rho_k$'s vanish. More precisely, we have the following lemma:

\begin{lemma}\label{constantresidue}
We have $\rho_k'=0$ if and only if $\rho_k\in \overline{\kp(x^p)}$. If moreover $\kp$ has characteristic zero or greater than $2d_x(d_y-1)$, then $\rho_k'=0$ if and only if $\rho_k\in \bar{\kp}$.
\end{lemma}

\begin{proof}
See for instance the proof of Lemma 2.4 in \citep{Gao}.
\end{proof}

Let us denote by $\kp[x,y]_{m,n}$ the vector space of bivariate polynomials of degree $\le m$ in $x$ and $\le n$ in $y$. We introduce the $\kp$-linear operator
\begin{eqnarray*}
\D: \kp[x,y]_{d_x,d_y-1} &\quad \longrightarrow &\quad \kp[x,y]_{3d_x-1,3d_y-3} \\
G & \quad \longmapsto &  \quad \Big(G_x F_y-G_y F_x\Big)F_y-\Big(F_{xy}F_y-F_{yy} F_x\Big)G,
\end{eqnarray*}
with the standard notations $F_y$, $F_{xy}$, etc. for the partial derivatives. Let $y_k(x)$ be the root of $F$ corresponding to the residue $\rho_k$. By combining the formulas 
$$
\rho_k(x)=\frac{G(x,y_k)}{F_y(x,y_k)} \quad {\rm and}\quad y_k'(x)=-\frac{F_x(x,y_k)}{F_y(x,y_k)},
$$ 
we are led to the equality
$$
\rho_k'(x)=\frac{\D(G)(x,y_k)}{F_y^3(x,y_k)},
$$
so that
$$
\rho_k'=0 \iff \D(G)(x,y_k)=0.
$$
In particular, since $F$ is separable with respect to $y$, it follows that $\rho_k'=0$ for all $k=1,\ldots,d_y$ if and only if $F$ divides $\D(G)$ in $\kp(x)[y]$. In order to reduce this division problem to a well estimated finite number of linear equations, we localize. 

Let $a\in \kp[x]$ be an irreducible polynomial which is coprime to the leading coefficient $\lc_y(F)\in \kp[x]$ of $F$. We denote by
$$
\ap:=\kp[x]_{(a)}
$$ 
the localization of $\kp[x]$ at $a$. Hence, euclidean division by $F$ in $\ap[y]$ is well defined. Each $Q\in \ap[y]$ has a unique  $a$-adic expansion
$$
Q(y)=\sum_{i=0}^{+\infty} q_i a^i,\qquad q_i\in \kp[x,y],\,\, \deg_x q_i < \deg\,a.
$$
For each pair of positive integers $0\le m\le n$, we introduce the truncated polynomial
$$
\big\{Q\big\}^{n}_m:=\sum_{i=m}^{n-1} q_i a^i.
$$
We have the following lemma, generalizing Lemma 3 in \citep{L}:

\begin{lemma}\label{divides}
Let $F$ and $\D(G)$ as before and denote by $\D(G)=QF+R$ the euclidean division of $\D(G)$ by $F$ in the ring $\ap[y]$. Let 
$$
m:=\Big\lfloor \frac{2d_x-1}{\deg\,a} \Big\rfloor +1\qquad{\rm and} \qquad n:=\Big\lceil \frac{3d_x-1}{\deg\,a}\Big\rceil+1.
$$
Then $F$ divides $\D(G)$ in $\kp(x)[y]$ if and only if $\big\{Q\big\}^{n}_m=\big\{R\big\}^{n}=0$.
\end{lemma}

\begin{proof}
Since the leading coefficient of $F$ is invertible in $\ap$, then $F$ divises $\D(G)$ in $\kp(x)[y]$ if and only if it divises $\D(G)$ in $\ap[y]$. Suppose that $F$ divises $\D(G)$ in $\kp(x)[y]$. Since both $F$ and $\D(G)$ lie in $\kp[x,y]$ and $F$ is primitive with respect to $y$, it follows from Gauss lemma that $F$ divises $\D(G)$ in $\kp[x,y]$. It follows that $R=0$ and that $Q\in\kp[x,y]$ has degree $\deg_x \D(G)-\deg_x F \le 2d_x-1$. In particular, the coefficients of $a^i$ in the $a$-adic expansion of $Q$ are zeroes as soon as $i\deg a > 2d_x-1$, that is for all $i\ge m$. In particular $\big\{Q\big\}^{n}_m=0$. Conversely, suppose that $\big\{Q\big\}^{n}_m=\big\{R\big\}^{n}=0$. Then the $a$-valuation of $\D(G)-\big\{Q\big\}^{m}F$ is greater or equal to $n$. It follows that
either $\D(G)=\big\{Q\big\}^{m}F$ or
$$
\deg _x (\D(G)-\big\{Q\big\}^{m}F) \ge n\deg a\ge 3d_x+a -1.
$$
But we have that
$$
\deg_x \big\{Q\big\}^{m}\le (m-1)\deg a+ \deg a-1\le 2d_x+\deg a-2
$$
so that $\deg_x (\D(G)-\big\{Q\big\}^{m}F)\le 3d_x+\deg a-2$. This forces equality $\D(G)=\big\{Q\big\}^{m}F$ in $\ap[y]$. Since all members of the equality lie in $\kp(x)[y]$, the equality holds in the ring $\kp(x)[y]$. 
\end{proof}

\begin{remark}\label{rem1}
If the leading coefficient of $F$ does not vanish at $0$, then we can take $a(x)=x$ and $\ap=\kp[[x]]$ in the previous lemma. More generally, if $\kp$ has cardinality greater than  $d_y$, we can reduce to that case up to replace $F(x,y)$ by $y^{d_y}F(x,\alpha+1/y)$ for some $\alpha\in \kp$ such that $F(0,\alpha)\ne 0$. This Moebius transformation has a negligeable cost for our purpose.  It only exchanges the points $(0,\infty)$ and $(0,\alpha)$ and does not modify the geometry of $F$ along the fiber $x=0$. 

If $\kp$ has cardinality $\le d_y$, then both degenerate situations  $u(0)=0$  and $F(0,\alpha)= 0$ for all $\alpha\in \kp$ might hold simultaneously. In such a case, we need to find a prime polynomial $a\in \kp[x]$ coprime to $u$. Note that we can find such an $a$
with $\deg a\in\mathcal{O}(\log(\deg u))$ (see the proof of Proposition \ref{algo1} in Section 4). 
\end{remark}

\begin{remark}
The situation $a(x)\ne x$ requires to perform euclidean division in $\ap[y]$ with $\ap\ne \kp[[x]]$. This may break the sparse structure of $\D(G)$ inherent to the sparse structure of the analytic factors $\F_i\in \kp[[x]][y]$. 
\end{remark} 

According to Lemma \ref{divides}, we introduce the following $\fp$-linear map:
\begin{eqnarray*}
\qquad \qquad\qquad\qquad\D_a:\fp^s & \longrightarrow & \kp[x,y]\times  \kp[x,y] \\
\mu & \longmapsto & \bigg(\frac{\{Q\}^{n}_m}{a^{m}},\big\{R\big\}^{n}\bigg)
\end{eqnarray*}
where $a,m,n$ are defined as in Lemma \ref{divides} and where $Q$ and $s$ are defined by the euclidean division $\D(G_{\mu})=QF+R$ in $\ap[y]$. 

\begin{corollary}\label{cor2}
We have equality $\ker(\D_a)=V(\overline{\kp(x^p)})$. If moreover $\kp$ has characteristic zero or greater than $2d_x(d_y-1)$, then $\ker(\D_a)=V(\bar{\kp})$.
\end{corollary}

\begin{proof}
Follows by combining Lemma \ref{constantresidue} and Lemma \ref{divides}.
\end{proof}

The image of $\D_a$ is contained in the finite-dimensional vector space
$$
\D_a(\fp^s)\subset \kp[x,y]_{d_x-1,2d_y-3}\times  \kp[x,y]_{3d_x-1,d_y-1}.
$$
Hence, for $\fp=\kp$, the computation of $\ker(\D_{a})$ reduces to compute the kernel of a $\kp$-linear system of $s$ unknowns and $\mathcal{O}(d_xd_y)$ equations.

\subsection{The case of small characteristic}\label{smallchar}

When the characteristic $p$ of $\kp$ is small, we need supplementary condition in order to know when $\rho_k\in \bar{\kp}$. In fact, we will get in such a case the stronger conditions $\rho_k\in \fp_p$ thanks to the following $\fp_p$-linear operator:
\begin{eqnarray*}
\qquad\qquad\qquad\qquad\N:\kp[x,y]_{d_x,d_y-1}&\longrightarrow &\kp[x,y^p]_{pd_x,d_y-1}\\
G &\longmapsto & G^p-\partial_y^{p-1}(GF^{p-1}).
\end{eqnarray*}
The operator $\N$ is well defined since $\partial_y(\N(G))=0$. The vector space $\kp[x^a,y^b]_{m,n}$ has to be understood as the vector space of polynomials of bidegree $(m,n)$ in the variables $(x^a,y^b)$. The operator $\N$ was introduced by Niederreiter in the context of univariate factorization over finite fields \citep{Nied}, and then used for bivariate factorization in \citep{L}. 

\begin{lemma}\label{VFp}
We have $V(\fp_p)=\{\mu\in \fp_p^s\,|\, \N(G_{\mu})=0\}$
\end{lemma}

\begin{proof}
This follows from Theorem $2$ in \citep{Nied}. 
\end{proof}

Since $\deg_x \N(G_{\mu})=pd_x$, the number of linear equations to be solved for computing $\ker(\N(G_{\mu}))$ grows linearly with $p$. The idea developed in \citep{L} is to combine $\N$ with the operator $\D$ in order to cut down this dependancy in $p$.

\begin{lemma}
Suppose that $\mu\in \ker(\D_a)$. Then $\N(G_{\mu})\in \kp[x^p,y^p]$. 
\end{lemma}

\begin{proof}
See \citep{L}, Lemma 4.
\end{proof}

Hence, previous lemma ensures that the $\fp_p$-linear map
\begin{eqnarray*}
\qquad\qquad\qquad\qquad\N_{a}:\ker(\D_a)&\longrightarrow &\kp[x^p,y^p]_{d_x,d_y-1} \\
\mu &\longmapsto& \N(G_{\mu})
\end{eqnarray*}
is well-defined. In particular, if $\kp=\fp_{p^k}$ is a finite field, the computation of $\ker(\D_{a})$ reduces to compute the kernel of a $\kp$-linear system of $\dim\ker(\D_a)\le r$ unknowns and $\mathcal{O}(kd_xd_y)$ equations over $\fp_p$. 

\section{Algorithms and complexity}\label{S6}

We combine now our previous results in order to give an algorithm for factorization and an algorithm for computing the number of rational factors, and we study their complexities.

\subsection{A rational factorization algorithm}

We obtain finally a deterministic algorithm for irreducible rational factorization of separable bivariate polynomials. The field $\fp$ now stands for $\kp$ if $\kp$ has characteristic zero or greater or equal to $d_x(d_y-1)$ and $\fp$ stands for the prime field $\fp_p$ of $\kp$ otherwise. Given a vector space $V$ over $\fp$, we denote by $\reb_{\fp}(V)$ the reduced echelon basis of $V$ over $\fp$. If we say compute $V$, this means compute $\reb_{\fp}(V)$.

\vskip4mm
\noindent
\textbf{Algorithm : Critical Factorization} 
\vskip4mm
\noindent
{\it \textbf{Input:} A bivariate polynomial $F\in \kp[x,y]$ separable with respect to $y$.
\vskip1mm
\noindent
\textbf{Output:} The irreducible rational factors of $F$.}

\vskip2mm
\noindent
\textbf{Step 0.} Compute the content $f\in \kp[x]$ of $F$ with respect to $y$ and do $F\leftarrow F/f$. Compute $f_1,\ldots,f_t$ the irreducible factors of $f$ over $\kp$.

\vskip2mm
\noindent
\textbf{Step 1.} Compute the truncated analytic factors $[\F_1]^{d_x+1},\ldots,[\F_s]^{d_x+1}$ of $F$. If $s=1$ then return $F$ is irreducible. Otherwise, build the polynomials $\big[\hat{\F_i}\partial_y\F_i]^{d_x+1}$ for all $i=1,\ldots,s$ and initialize $S_0\leftarrow \fp^s$.
\vskip2mm
\noindent
\textbf{Step 2.} If $\lc_y(F)(0)\ne  0$, let $a\leftarrow x$. Otherwise, compute $a\in \fp[x]$ of degree $a\in \mathcal{O}(\log d_x)$ such that $a$ is irreducible and coprime to $u$. 
\vskip2mm
\noindent
\textbf{Step 3.} Build the $\fp$-linear system associated to $\D_a$ and compute $S_0\leftarrow \ker(\D_a)$. If $\dim_{\fp}S_0=1$, return $F$ is irreducible. If $\kp$ has characteristic zero or greater than $d_x(d_y-1)$, then go to Step 5. Else go to Step 4.
\vskip2mm
\noindent
\textbf{Step 4.} Build the $\fp$-linear system associated to $\N_a$ and compute $S_0\leftarrow S_0\cap \ker(\N_a))$. If $\dim_{\fp}S_0=1$, return $F$ is irreducible. If $\reb(S_0)$ does not form a partition of $(1,\ldots,1)$ or if $Z\ne 0$ then go to Step 5. Otherwise, go to Step 7.
\vskip2mm
\noindent
\textbf{Step 5.} Compute $q:=\lfloor v/d\rfloor$ where $v:=\val_x(\Disc_y(F))$ and $d$ is the minimal $y$-degree of the $[\F_i]^{d_x+1}$. If $q\le d_x$, go to Step 7, otherwise go to Step $6$.
\vskip2mm
\noindent
\textbf{Step 6.} Compute the $q$-truncated analytic factors of $F$ and compute $S_0 \leftarrow S_0\cap W^{q}$.
\vskip2mm
\noindent
\textbf{Step 7.} We have $S=S_0$. For each $v_j\in \reb(S)$, $j=1,\ldots,r$, compute 
$$
\tilde{F}_j:=\Big[\lc(F)\prod_{i=1}^s \F_i^{v_{ji}}\Big]^{d_x+1}
$$
and compute the primitive part $F_j$ of $\tilde{F}_j$ with respect to $y$. 
\vskip2mm
\noindent
\textbf{Step 8.} Return $f_1,\ldots,f_t,F_1,\ldots, F_r$.

\vskip4mm
\noindent
The proof of Theorem \ref{t2} follows from the following proposition.

\begin{proposition}\label{algo1}
The algorithm Critical Factorization is correct. It performs at most one univariate factorisation in $\kp[x]$ of degree $d_x$ and :

- If $\kp$ has characteristic zero or greater than $d_x(d_y-1)$, then at most
$$
\mathcal{O}(d_xd_y^2+d_y\max(d_x,q) s^{\omega-1})+\C(\max(d_x,q))
$$
arithmetic operations over $\kp$. 

- If $\kp=\fp_{p^k}$, then at most $$\wt{\mathcal{O}}(kd_xd_y^2+kd_y\max(d_x,q) s^{\omega-1})+\sO(k)\C(\max(d_x,q))$$ operations over $\fp_p$.
\end{proposition}

\begin{proof}
We have $S_0=V(\fp_p)$ at the end of Step 4 by Lemma \ref{constantresidue}. Hence $S=S_0$ if and only if $Z=0$ by Proposition \ref{NumFac}. If $\reb(S_0)$ does not form a partition of $(1,\ldots,1)$ or $Z\ne 0$, then $S\ne S_0$ and we need to go to Step $5$. Otherwise, the recombination problem is solved and we can go directly to Step 7. 

We have $S_0=V(\bar{\kp})$ or $S_0=V(\fp)$ at the beginning of Step 5. 
Since $\deg_x \lc(\F_i)\le d_x$, we have that $\deg_y \F_i=\deg_y [\F_i]^{d_x+1}$. Hence $q\ge N$ by the Remark \ref{remtrunc}. It follows from Theorem \ref{critic} and Corollary \ref{cor1} that $S=S_0$ at Step 7. For all $j=1,\ldots,r$, we have that $\tilde{F}_j= \frac{\lc(F)}{\lc(F_j)}F_j$ since 
$$\deg_x(\lc(F)/\lc(F_j)) + \deg_x (F_j ) \le \deg_x F.$$
Hence, the algorithm returns a correct answer. Let us study its complexity.

\textbf{Step 0.} Computation of the content of $F$ requires $\tilde{\sO}(d_xd_y)$ arithmetic operations over $\kp$.

\textbf{Step 1.} Computations of the $(d_x+1)$-truncated $\F_i$'s has complexity $\C(d_x)$ by definition. Then we compute all $\big[\hat{\F_i}\partial_y\F_i]^{d_x+1}$ with $\tilde{\sO}(sd_xd_y)$ operations in $\kp$.

\textbf{Step 2.} If $\kp$ has characteristic zero or greater than $d_x$, we can take $a=x-c$ for some $c$ such that $\lc(F)(c)\ne 0$. Otherwise, a basic approach (certainly not the most efficient) consists to remark that for $n>\log_p (\deg_x \lc(F))$, the polynomial
$$
\tilde{a}_{n}:=\frac{x^{p^n}-x}{\gcd(x^{p^n}-x,\lc(F))}
$$
has positive degree and is coprime to $u$. Then, we take for $a$ an irreducible factor of $\tilde{a}_{n}$, which has necessary degree less or equal to $n\in \mathcal{O}(\log d_x)$. Thanks to fast gcd computations, we compute $\tilde{a}_{n}$ within $\tilde{\mathcal{O}}(\max(\deg_x \lc(F), p^n))\subset \tilde{\mathcal{O}}(d_x)$ operations. Then, we need to perform a univariate factorization of degree at most $d_x$ in order to find $a$.

\textbf{Step 3.} In order to build the linear system of the map $\D_a$, we need first to compute $D_i:=\D(G_{e_i})$ for $e_i$ varying over the canonical basis of $\fp^s$. This has complexity $\tilde{\sO}(sd_xd_y)$. Then, we need to compute the $a$-adic expansion of the $D_i$'s and $F$. This costs $\tilde{\sO}(sd_xd_y)$ thanks to \citep{GG}, Theorem 9.15. We need then to perform $s$ euclidean divisions  in $(\ap/a^m)[y]$ of polynomials of degree $\sO(d_y)$, and where $m\in \sO(d_x/a)$. This costs $\tilde{\sO}(sd_y)$ operations in $\ap/a^m$, hence $\tilde{\sO}(sd_xd_y)$ operations over $\kp$.
Then computing the reduced echelon basis of $\ker(\D_a)$ requires $\sO(d_xd_y s^{\omega-1})$ operations in $\kp$.

\textbf{Step 4.} We compute here $S_0$ with at most $\wt{\sO}(kd_xd_ys^{\omega-1})$ arithmetic operations over $\fp_p$ thanks to \citep{L}, Proposition 4. It has to be noticed that the construction of the linear system associated to $\N_a$ might constitute a bottelneck to the algorithm. The $\fp$-vector space $Z$ is determined by $s$ unknowns and $\sO(d_xd_y)$ equations over $\kp$, hence testing $Z\ne 0$ requires at most  $\mathcal{O}(d_xd_ys^{\omega-1})$ operations over $\kp$.  Note that $Z=0$ as soon as $F$ is locally irreducible along $x=0$ (Lemma \ref{trivZ}).

\textbf{Step 5.} Since $v\le 2d_xd_y$, we can compute $\Disc_y(F)$ by considering $F$ as a polynomial in $y$ with coefficient in the ring $\kp[x]/(x^{2d_xd_y})$. This requires $\wt{\mathcal{O}}(d_xd_y^2)$ arithmetic operations over $\kp$ thanks to fast euclidean algorithm \citep{GG}. Computing $q$ then has a negligeable cost.

\textbf{Step 6.} Computing the higher truncations of the $\F_i$'s requires $\C(N)$ operations over $\kp$ by definition. Then computing all $\big[\hat{\F_i}\partial_y\F_i]^{N}
_{d_x+1}$ requires $\tilde{\sO}(sNd_y)$ operations in $\kp$. Building the linear system that determines the equations of $W^N\cap S_0$ requires at most $\tilde{\sO}(Nd_y)$ operations in $\kp$ thanks to the sub-product tree technique. Then computing the reduced echelon basis of $W^N\cap S_0$ requires 
at most $\sO((N-d_x)d_yt^{\omega-1})$ operations over $\kp$, where $t=\dim_{\kp}U\le s$.

\textbf{Step 7.} Let $n_j:=\deg_y F_j$. Then the computation of $\tilde{F}_j$ requires $\tilde{\sO}(d_x n_j)$ operations in $\kp$, hence a total of $\tilde{\sO}(d_xd_y)$ operations for all $j$. The cost of primitive parts computations amounts to the same number of operations.

The proof follows by adding all these costs, and by remarking that one arithmetic operation over $\fp_{p^k}$ requires $\sO(k)$ operations over $\fp_p$.
\end{proof}


\begin{remark} \textit{(About Step 5.)} The cost of the discriminant computation might be very high in terms of bit complexity. Moreover, it might happen that $N$ and $q$ differ by a factor of the order of magnitude of $d_y$. Hence it is much more preferable to approximate the $q_i$'s (hence $N$) during the computation of the $\F_i$'s, for instance by using the relations with the characteristic Puiseux exponents of the branches $\F_i$ (see for instance \citep{PR}).  Hence, the cost of Step 5 is included in $\C(N)$.
Note that if the reduced echelon basis of $V(\bar{\kp})$ does not form a partition of $(1,\ldots,1)$ or if $Z\ne 0$, then necessarily $N>d_x$. Finally, note that if $\kp=\qp$ (or more generally a number field), there are strategies to compute the valuation of the discriminant by working modulo a well chosen prime $p$ \citep{PR}.
\end{remark}

%
%

\subsection{Algorithms for the number of irreducible rational factors}

In the first algorithm, we suppose that $\kp$ has positive characteristic, or that $\kp$ is a decomposition field of $F$. The field $\fp$ stands for $\kp$ if $\kp$ has characteristic zero or for the prime field $\fp_p$ of $\kp$ otherwise. 

\vskip4mm
\noindent
\textbf{Algorithm : Number of Factors} 
\vskip4mm
\noindent
{\it \textbf{Input:} A bivariate polynomial $F\in \kp[x,y]$ separable with respect to $y$.
\vskip1mm
\noindent
\textbf{Output:} The number of irreducible rational factors.}

\vskip2mm
\noindent
\textbf{Step 0.} Compute the content $f\in \kp[x]$ of $F$ with respect to $y$ and do $F\leftarrow F/f$. Compute $t$ the number of irreducible factors of $f$ over $\kp$.

\vskip2mm
\noindent
\textbf{Step 1.} Compute the $d_x+1$-truncated analytic factors of $F$. 
\vskip2mm
\noindent
\textbf{Step 2.} Compute $r\leftarrow\dim V(\fp)$. If $r=1$ then return $t+1$.
\vskip2mm
\noindent
\textbf{Step 3.} Compute $r\leftarrow r - \dim Z$.  Return $t+r$.

\vskip4mm
\noindent

\begin{proposition}\label{algo2}
The algorithm Number of Factors  is correct. It performs at most one univariate factorisation in $\kp[x]$ of degree $d_x$ and :

- If $\kp$ has characteristic zero, then 
$
\mathcal{O}(d_x d_y s^{\omega-1})+\C(d_x)
$
operations over $\kp$. 

- If $\kp=\fp_{p^k}$, then $\mathcal{O}(kd_x d_y s^{\omega-1})+\sO(k)\C(d_x)$  operations over $\fp_p$.
\end{proposition}

\begin{proof}
The correctness of the algorithm follows from Proposition \ref{NumFac}. In positive characteristic, we can compute $V(\fp)$ as in the previous algorithm and the complexity analysis follows from the proof of previous proposition \ref{algo1}. In characteristic zero, we have that $V(\fp)=V(\bar{\kp})$ by Lemma \ref{KKbar} so that we compute $V(\fp)$ as in Step 3 of algorithm Critical Factorization. The $\fp$-vector space $Z$ is determined by $s$ unknowns and $\sO(d_xd_y)$ equations over $\kp$, hence testing $Z\ne 0$ requires at most  $\mathcal{O}(d_xd_ys^{\omega-1})$ operations over $\kp$. 
\end{proof}
\vskip2mm
\noindent

Finally, we have the following algorithm for an irreducibility test. Here, no hypothesis are made on the field $\kp$. 
\vskip4mm
\noindent
\textbf{Algorithm : Irreducibility Test} 
\vskip4mm
\noindent
{\it \textbf{Input:} A bivariate polynomial $F\in \kp[x,y]$ separable with respect to $y$.
\vskip1mm
\noindent
\textbf{Output:} True if $F$ is irreducible over $\kp$, False otherwise.}

\vskip2mm
\noindent
\textbf{Step 0.} If $F$ is not primitive with respect to $y$, then return False. Otherwise, replace $F$ by its primitive part. 
\vskip2mm
\noindent
\textbf{Step 1.} Compute the $n$-truncated analytic factors of $F$ with $n=d_x+1$ if $\Char(\kp)>0$ and $n=2d_x$ otherwise.
\vskip2mm
\noindent
\textbf{Step 2.} Compute $r\leftarrow \dim V(\fp)$ if $\Char(\kp)>0$ and $r\leftarrow \dim V(\bar{\kp})\cap W^{2dx}$ otherwise. If $r=1$ return True.
\vskip2mm
\noindent
\textbf{Step 3.} Compute $r\leftarrow r-\dim Z$ if $\Char(\kp)>0$ and $r\leftarrow r -\dim Z\cap W^{2d_x}$ otherwise. 
\vskip2mm
\noindent
\textbf{Step 4.} Return True if $r=1$, False otherwise. 

\vskip4mm
\noindent
\begin{proposition}\label{algo3}
The algorithm Irreducibility Test  is correct. It performs at most one univariate factorisation in $\kp[x]$ of degree $d_x$ and :

- If $\kp$ has characteristic zero, then 
$
\mathcal{O}(d_x d_y s^{\omega-1})+\C(2d_x)
$
operations over $\kp$. 

- If $\kp=\fp_{p^k}$, then  $\mathcal{O}(kd_x d_y s^{\omega-1})+\sO(k)\C(d_x)$ operations over $\fp_p$.
\end{proposition}

\begin{proof}
The correctness of the algorithm follows from Proposition \ref{NumFac} in positive characteristic and from Proposition \ref{propirr} otherwise. The complexity analysis follows from the proof of the two previous algorithms.
\end{proof}

The proof of Theorem \ref{t1} follows from the two previous propositions \ref{algo2} and \ref{algo3}.

\subsection{Locally irreducible polynomials}\label{sslocirr}

We recall from the introduction that we say that $F$ is locally irreducible along the line $x=0$ (resp. absolutely locally irreducible) if the germs of curves $(C,P)\subset (\mcpp_{\kp}^2,P)$ defined by $F$ are irreducible over $\kp$ (resp. over $\bar{\kp}$) at each rational place $P$ of the line $x=0$, including the place at infinity.

\begin{lemma}\label{trivZ} Suppose that $\kp$ has characteristic greater or equal to $d_y$. If $F$ is locally irreducible along the line $x=0$, then $Z=0$.
\end{lemma}

\begin{proof}
Let $\mu\in Z$. we have in particular 
\begin{equation}\label{relation}
\sum_{i=1}^s \mu_i\hat{\F_i}\partial_y \F_i(0,y)=0.
\end{equation}
Let us consider first an analytic factor $\F_i$ corresponding to affine place of $C$ along $x=0$. Since $F$ is locally irreducible along the line $x=0$. It follows from Hensel's lemma \citep{GG} that $\F_i(0,y)$ is a power of a prime polynomial that is coprime to $\hat{\F_i}(0,y)$. Hence relation  $(\ref{relation})$ combined with Gauss Lemma imposes that $\F_i(0,y)$ divises $\mu_i\partial_y \F_i(0,y)$, hence $\mu_i=0$ thanks to the assumption on the characteristic of $\kp$. Suppose now that $F$ has an analytic factor, say $\F_1$ that vanishes at $(0,\infty)$. By the local irreducibility assumption, $\F_1$ is the unique such factor. Hence $\mu\in Z$ becomes equivalent to that 
$$
\mu_2=\cdots=\mu_s=0\quad {\rm and}\quad \mu_1[\hat{\F_1}\partial_y(\F_1)]^{d_x+1}=0,
$$
thanks to what we proved for the affine places. The leading coefficient of $\hat{\F_1}\partial_y(\F_1)$ is equal to $d_1\lc_y(F)$. It has degree $\le d_x$, and $d_1\ne 0$ by assumption on $\kp$. Hence, last equation implies $\mu_1=0$. 
\end{proof}

\begin{lemma}\label{lirr}
If $\Card(\kp)>d_y$ and $F$ is locally irreducible along the line $x=0$, then we compute the $(d_x+1)$-truncated analytic factors of $F$ with one univariate factorization of degree at most $d_y$ plus $\wt{\sO}(d_xd_y)$ arithmetic operations over $\kp$.
\end{lemma}

\begin{proof}
Since $\Card(\kp)>d_y$, there exists $y_0\in \kp$ such that $F(0,y_0)\ne 0$. To find such an element $y_0$ has a negligeable cost once the univariate factorization of $F(0,y)$ is given.  Then the Moebius transformation $F\leftarrow y^{d_y}F(x,\alpha+1/y)$ reduce to the case where $u=\lc_y(F)$ is a unit modulo $x$. Hence we are in position where
$$
F(0,y)=u(0)\prod_{i=1}^s P_i^{m_i}, \quad \F_i(0,y)=P_i^{m_i}
$$
for some distinct prime polynomials $P_i\in \kp[y]$. We can lift this factorization  modulo $x^{d_x+1}$ with $\wt{\sO}(d_xd_y)$ arithmetic operations over $\kp$ thanks to the multifactor Hensel lifting \citep{GG}, Theorem 15.18. Then, we perform the inverse Moebius transformation in order to get the analytic factors of the original polynomial. Both Moebius transformations can be done in softly optimal time $\wt{\sO}(d_x d_y)$ with interpolation/evaluation.
\end{proof}

\paragraph*{Proof of Theorem \ref{tirr}.} It follows from Lemma \ref{trivZ} and Proposition \ref{NumFac} that $S=V(\kp)$ under the assumption of Theorem \ref{tirr}. If 
$\kp=\fp_{p^k}$ is a finite field, we can compute a basis of $V(\kp)$ from the $(d_x+1)$-truncated factors within $\wt{\sO}(kd_xd_y s^{\omega-2})$ operations over $\fp$ (see the proof of Proposition \ref{algo1}). Since one operation in $\kp$ amounts to $\sO(\kp)$ operations over $\fp$, the proof follows from Lemma \ref{lirr}. If $\kp$ is a decomposition field for $F$, we have $V(\kp)=V(\bar{\kp})$ thanks to Lemma \ref{KKbar}. Since $\kp$ has characteristic zero or greater than $d_x(2d_y-1)$, we compute the reduced echelon basis of $V(\bar{\kp})$ within $\wt{\sO}(d_xd_y s^{\omega-1})$ operations over $\kp$. $\hfill\square$

\section{Absolute factorization}\label{Sabs}

In order to generalize our results to the absolute case, we follow the strategy developed by Ch\`eze-Lecerf in the regular case \citep{CL}.

\paragraph*{Absolute factorization.}

We denote by $E_1,\ldots,E_{\bar{r}}$ the irreducible factors of $F$ in $\bar{\kp}[x,y]$. We represent this factorization by a family of pairs of polynomials
$$
\{(P_1,q_1),\ldots,(P_{t},q_{t})\}
$$
where $q_k\in \kp[z]$ is monic, $\deg_z P_k < \deg q_k$, $P_k(x,y,\phi)\in \bar{\kp}[x,y]$ has constant bidegree when $\phi$ runs over the roots of $q_k$, and for each factor $E_i$ there exists a unique pair $(k,\phi)$ such that $q_k(\phi)=0$ and 
$$
E_i(x,y)=P_k(x,y,\phi).
$$
Such a representation is not unique, but is not redundant. The $q_k$'s are irreducible if and only if the products $\prod_{q_k(\phi)=0} P_k(\phi)$ are the irreducible factors of $F$ in $\kp[x,y]$.

\paragraph*{Absolute analytic factorization.}
We denote by $\E_1,\ldots,\E_{\bar{s}}$ the irreducible analytic factors of $F$ in $\bar{\kp}[[x]][y]$. As before, we suppose that the $\E_i$'s are given by a collection of pairs of polynomials
$$
\{(\mcp_1,p_1),\ldots,(\mcp_{\ell},p_{\ell})\}
$$
where $p_k\in \kp[z]$ is monic, $\deg_z \mcp_k < \deg p_k$, $\mcp_k(\phi)\in \bar{\kp}[[x]][y]$ has constant degree in $y$ when $\phi$ runs over the roots of $p_k$, and for each $\E_i$ there is a unique pair $(k,\phi)$ such that $p_k(\phi)=0$ and 
$$
\E_i(x,y)=\mcp_k(x,y,\phi).
$$
In particular, the $p_k$'s are separable. The $p_k$'s are irreducible if and only if ${\ell}=s$, if and only if  $$(\deg_y(\mcp_k),\deg_z(p_k))=(e_k,f_k),$$ where $e_k$ and $f_k$ stand for the ramification index and residual degree at the rational places of $F=0$ over $x=0$. We do not necessarily assume this here, the only important point from a complexity point of view being that we necessarily have
$$
\bar{s}=\sum_{k=1}^{\ell} \deg_z(p_k)=\sum_{i=1}^{s} f_i.
$$
In particular, the more the curve $F=0$ is ramified over $x=0$, the smaller the number of unknowns is. This is a great difference with the regular case, for which equality $\bar{s}=d_y$ always holds. We call the $n$-truncated absolute analytic factorization the data of the pairs $([\mcp_k]^{n+1},p_k)$.

\paragraph*{Solving recombinations with $\bar{\kp}$-linear algebra.}

In analogy to the rational case, we denote by 
$$
\bar{S}=\langle \bar{v}_1,\ldots,\bar{v}_{\bar{r}}\rangle_{\bar{\kp}} \subset \bar{\kp}^{\bar{s}}
$$ the $\bar{\kp}$-vector space generated by the recombination vectors $\bar{v}_1,\ldots,\bar{v}_s$ solution to 
$$
E_j=\bar{u}_j\prod_{i=1}^{\bar{s}} \E_i^{\bar{v}_{ji}}, \quad j=1,\ldots,\bar{r},
$$
with $\bar{u}_j\in \kp[x]$, $\bar{u}_j(0)=1$. For $\mu\in \bar{\kp}^{\bar{s}}$, we denote by 
$$
\G_{\mu}:=\sum_{i=1}^{\bar{s}} \mu_i \hat{\E}_i \partial_y \E_i \in \bar{\kp}[[x]][y].
$$
We introduce the $\bar{\kp}$-vector spaces
$$
\bar{V}:=\Big\{\mu\in \bar{\kp}^{\bar{s}}\,\, | \,\, \,[\G_{\mu}]^{d_x+1}\in \langle \hat{E}_1 \partial_y E_1,\ldots,E_{\bar{r}}\partial_y E_{\bar{r}}\rangle\Big\}.
$$
Hence $\mu\in\bar{V}$ if and only if the residues of $[\G_{\mu}]^{d_x+1}/F$ lie in $\bar{\kp}$ by Lemma \ref{gao}. We introduce also
$$ 
\bar{Z}:=\big\{\mu\in \bar{\kp}^{\bar{s}}\,\, | \, \,[\G_{\mu}]^{d_x+1}=0\big\},
$$
and
$$
\bar{W}:=\big\{\mu\in \bar{\kp}^{\bar{s}}\,\, | \, \,[\G_{\mu}]^{N}_{d_x+1}=0\big\}
$$
where $N$ is the separability order of $F$. Lemma \ref{gao} combined with Theorem \ref{critic} and Proposition \ref{NumFac} give the relations
$$
\bar{S}=\bar{V}\cap \bar{W}\quad {\rm and}\quad  \bar{V}=\bar{S}\oplus \bar{Z}.
$$
First equality leads to an algorithm for solving recombinations. The second equality leads to an algorithm for computing the number of irreducible absolute factors. The idea now is to use the Vandermonde matrices attached the polynomials $p_k$'s in order to compute a basis of the involved vector spaces with linear algebra over $\kp$.


\paragraph*{The Vandermonde isomorphism.}
We define the partition
$\{1,\ldots,\bar{s}\}=I_1\cup\cdots \cup I_{\ell}$ 
by requiring that 
$$
\prod_{i\in I_k} \E_i(x,y) = \prod_{p_k(\phi)=0} \mcp_k(x,y,\phi), \quad\,k=1,\ldots \ell.
$$
These products lie in $\kp[[x]][y]$. They coincide with the irreducible analytic factors of $F$ if and only if the $p_k$'s are irreducible.  Let $n_k:=\deg p_k$. If $\mu\in \bar{\kp}^{\bar{s}}$, we denote by 
$\mu^{(k)} \in \bar{\kp}^{n_k}$ 
the vector whose entries are the entries of $\mu$ whose index lie in $I_k$. We introduce the $\bar{\kp}$-linear map
\[\begin{aligned}
A=(A_1, \ldots ,A_{\ell}):\bar{\kp}^{\bar{s}} & \longrightarrow  \kp^{\bar{s}}\otimes_{\kp} \bar{\kp} \\
\mu=(\mu^{(1)},\ldots,\mu^{(\ell)}) & \longmapsto \nu:=(\nu^{(1)},\ldots,\nu^{(\ell)})
\end{aligned}\]
where $A_k$ stands for the transposed of the Vandermonde matrix of the roots $(\phi_{k1},\ldots,\phi_{k n_k})$ of $p_k$. In other words, the vector
$$
\nu^{(k)}:=A_k \, \mu^{(k)}\in \kp^{n_k}\otimes\bar{\kp}
$$
is defined by
$$
\nu^{(k)}_j:=\sum_{i=1}^{n_k} \mu^{(k)}_{i} \phi_{ki}^{j-1}.
$$
Since the polynomials $p_k$ are separable, each map $A_k:\bar{\kp}^{n_k}\to\kp^{n_k}\otimes\bar{\kp}$ is an isomorphism, hence so is the map $A$.


\paragraph*{Recombinations over $\kp$.}
For a given $\G\in \kp[[x]][y,z]$, we denote by $\coef(\G,z^{j})\in \kp[[x]][y]$ the coefficient of $z^{j}$ in $\G$. We have by construction that $\mcp_k$ divises $F$ in the ring $\kp[z]/(p_k)[[x]][y]$ and we denote by $\hat{\mcp}_k$ the unique polynomial such that $F=\mcp_k \hat{\mcp_k}\in \kp[[x]][y,z]$, with $\deg_z \hat{\mcp}_k<\deg p_k$. Given $\nu \in \bar{\kp}^{\bar{s}}$, we denote by
$$
\HH_{\nu}:=\sum_{k=1}^{\ell} \sum_{j=1}^{n_k} \nu^{(k)}_{j} \coef(\hat{\mcp}_k\partial_y \mcp_k,z^{j-1}).
$$
We introduce the $\kp$-vector spaces
$$
V_{\kp}:=\Big\{\nu\in \kp^{\bar{s}}\,\, | \,\, \,[\HH_{\nu}]^{d_x+1} \in \langle \hat{E}_1 \partial_y E_1,\ldots,E_{\bar{r}}\partial_y E_{\bar{r}}\rangle_{\kp}\Big\}
$$
$$
Z_{\kp}:=\Big\{\nu\in \kp^{\bar{s}}\,\, | \, \,[\HH_{\nu}]^{d_x+1}=0\Big\}.
$$
and 
$$
W_{\kp}:=\Big\{\nu\in \kp^{\bar{s}}\,\, | \, \,[\HH_{\nu}]^N_{d_x+1}=0\Big\},
$$
We have the following
\begin{proposition}\label{vvbar}
The isomorphism $A:\mu \mapsto \nu$
induces isomorphisms 
$$\bar{V}=V_{\kp}\otimes_{\kp} \bar{\kp},\quad \bar{W}=W_{\kp}\otimes_{\kp} \bar{\kp}\quad {\rm and} \quad \bar{Z}=Z_{\kp}\otimes_{\kp} \bar{\kp}.$$
\end{proposition}

\begin{proof}
By construction, we have that
\[\begin{aligned}
\G_{\mu}=\sum_{k=1}^{\ell} \sum_{i=1}^{n_k} \mu^{(k)}_i \hat{\mcp}_k\partial_y \mcp_k(\phi_{ki}) &=\sum_{k=1}^{\ell} \sum_{i=1}^{n_k} \sum_{j=1}^{n_k} \coef (\hat{\mcp}_k\partial_y \mcp_k,z^{j-1})\phi_{ki}^{j-1}\\
&=\sum_{k=1}^{\ell} \sum_{j=1}^{n_k} \Big(\sum_{i=1}^{n_k} \phi_{ki}^{j-1}\Big)\coef (\hat{\mcp}_k\partial_y \mcp_k,z^{j-1})\\
&=\sum_{k=1}^{\ell} \sum_{j=1}^{n_k} \nu^{(k)}_j\coef (\hat{\mcp}_k\partial_y \mcp_k,z^{j-1})=\HH_{\nu}.
\end{aligned}\]
The claimed isomorphisms then follow from the definitions of the involved vector spaces. 
\end{proof}

\paragraph*{Proof of Theorem \ref{abs1}.}

We have shown how to compute a basis of all involved vector spaces with linear algebra over $\kp$. Unfortunately, we don't have the reduced echelon basis trick when working with the unknowns $\nu$ instead of $\mu$. To solve this problem, we rather use an absolute partial fraction decomposition algorithm along a regular fiber, following Section 4 in \citep{CL}. We obtain the following algorithm.

\vskip4mm
\noindent
\textbf{Algorithm : Absolute Factorization} 
\vskip4mm
\noindent
{\it \textbf{Input:} A field $\kp$ with cardinality at least $d_x(2d_y-1)$ and a bivariate polynomial $F\in \kp[x,y]$ separable with respect to $y$.
\vskip1mm
\noindent
\textbf{Output:} A family $\{(P_1,q_1),\ldots,(P_t,q_t)\}$ that represents the absolute factorization of $F$.}

\vskip2mm
\noindent
\textbf{Step 1.} Compute a basis  $\nu_1,\ldots,\nu_{\bar{r}}$ of $V_{\kp}\cap W_{\kp}$.
\vskip2mm
\noindent
\textbf{Step 2.} Find a regular fiber $x=\alpha$ for some $\alpha\in \kp$.
\vskip2mm
\noindent
\textbf{Step 3.} Compute $h_1:=\HH_{\nu_1}(\alpha,y),\ldots,h_{\bar{r}}:=\HH_{\nu_{\bar{r}}}(\alpha,y)$.
\vskip2mm
\noindent
\textbf{Step 4.} Call Algorithm 7 in \citep{CL} in order to find $(c_1,\ldots,c_{\bar{r}})\in \kp^{\bar{r}}$ that separate the residues of the $h_i$'s. 
\vskip2mm
\noindent
\textbf{Step 5.} Let $h=\sum c_i h_i$. Call the Lazard-Rioboo-Trager algorithm
(Algorithm 14 in \citep{CL}) in order to compute the absolute partial fraction decomposition of $h/F(\alpha,y)$.
\vskip2mm
\noindent
\textbf{Step 6.} Call Algorithm 6 in \citep{CL} of absolute multi-factor Hensel lifting in order to lift the decomposition of Step 6 to $\{(P_1,q_1),\ldots,(P_t,q_t)\}$.

\begin{proposition}\label{algoabs}
Let $m=\max(d_x+1,N)$ and $p=\Char(\kp)$. The algorithm Absolute Factorization is correct. It performs at most 
$$
\sO(\bar{s}^{\omega-1}\max(d_x+1,N)d_y+\bar{r}d_x d_y^2)+\C(\max(d_x+1,N))
$$ operations in $\kp$ if $p=0$ or $p > dx(2d_y-1)$ and at most 
$$
\wt{\sO}(k\bar{s}^{\omega-1}\max(d_x+1,N)d_y+k\bar{r}d_x d_y^2)+\sO(k)\C(\max(d_x+1,N))
$$ operations over $\fp_p$ if $\kp=\fp_{p^k}$.
\end{proposition}

\begin{proof}
Given the $\mcp_k$'s and the $p_k$'s, we can compute a basis of the $\kp$-vector spaces $V_{\kp}$ and $W_{\kp}$ with the the same cost as in the rational case, with the number of unknowns $s$ being replaced by $\bar{s}$.  
Given $(\nu_1,\ldots,\nu_{\bar{r}})$ a basis of $V_{\kp}\cap W_{\kp}$, we have by construction that
$$
\langle \HH_{\nu_1},\ldots,\HH_{\nu_{\bar{r}}}\rangle_{\bar{\kp}} = \langle  \hat{E_1}\partial_y E_1,\ldots \hat{E}_{\bar{r}}\partial_y E_{\bar{r}} \rangle_{\bar{\kp}},
$$
with $\HH_{\nu_i}\in \kp[x,y]$. By assumption on the cardinality of the field, we know that there exists a regular fiber $x=\alpha$ over which $F(\alpha,y)$ is separable of degree $d_y$. We can find such a fiber by computing $\Disc F(i,y)$ for $i=1,\ldots ,d_x(2d_y-1)+1$ until we reach a non vanishing discriminant. This costs at most $\sO(d_xd_y^2)$ operations over $\kp$. Then we refer to \citep{CL}, Paragraph 4 for the remaining steps of the algorithm. Step $4$ costs $\sO(\bar{r}d_x d_y^2)$ operations in $\kp$, step $5$ costs $\sO(d_y^2)$ operations in $\kp$ and step $6$ costs $\wt{\sO}(d_x d_y)$ operations in $\kp$. 
\end{proof}

Theorem \ref{abs1} follows immediately from the previous proposition.

\paragraph*{Proof of Theorem \ref{abs2}.}

We have by Proposition \ref{vvbar} that the number of absolutely irreducible factors of $F$ is equal to 
$$
\bar{r}=\dim_{\bar{\kp}}\bar{S} =\dim_{\kp} V_{\kp} -\dim_{\kp} Z_{\kp}
$$
Given the $d_x+1$-truncated analyic factorization of $F$, we can compute $\dim_{\kp} V_{\kp}$ and $\dim_{\kp} Z_{\kp}$ with the same costs as in the rational case, with the number of unknowns $s$ being replaced by $\bar{s}$. The proof of Theorem \ref{abs2} follows. $\hfill\square$

\section{Non degenerate polynomials}\label{Spolytope}

We introduce the notion of $P$-adic Newton polytopes. These combinatorial objects give a lot of interesting informations for both rational and analytic factorization. In particular, we show here that they permit to detect a large class of polynomials whose separability order is small.


Let us fix $P\in \kp[y]$ a non constant polynomial. Any polynomial $\F\in \kp[[x]][y]$ can be uniquely expanded as 
$$
\F(x,y)=\sum f_{ij} x^i P^j\in \kp[[x]][y],
$$
with $f_{ij}\in \kp[y]$, $\deg f_{ij} < \deg P$. Let
$$
\Supp_P(\F):=\big\{(i,j)\in \np^2,\,\, f_{ij}\ne 0\big\}
$$
stands for the $P$-support of $\F$. The $P$-adic Newton polytope of $\F$, or $P$-polytope for short, is the convex hull of the positive cone generated by the support of $\F$, that is 
$$
\N_{P,\F}:=\Conv \Big((\Supp_P(\F)+(\rp^+)^2\Big).
$$
When $P=y$ we recover the usual notion of Newton polytope of a bivariate power series \citep{Kou}, and we might simply say Newton polytope for the $y$-polytope. Take care that the terminology of Newton polytope refers sometimes in the litterature for the (compact) convex hull of the support of a bivariate polynomial. We have the following lemma. 

\begin{lemma}\label{lempol}
Let $P\in \kp[y]$ be separable and irreducible and let $\alpha\in \bar{\kp}$ be a root of $P$. Then, the $P$-polytope of $\F$ coincides with the Newton polytope of $\F(x,y+\alpha)$.
\end{lemma}

\begin{proof}
Since $P$ is irreducible and separable, the highest power of $P$ that divises a given $f\in \kp[y]$ coincides with the highest power of $y-\alpha$ that divises $f$ in $\bar{\kp}[y]$, which coincides with the highest power of $y$ that divises $f(y+\alpha)$. The Lemma follows. 
\end{proof}

We call the $P$-edges of $\F$ the compact edges of its $P$-polytope. Let $\Lambda$ be a $P$-edge and let $a_{\Lambda}$ and $b_{\Lambda}$ stand respectively for the distance from $\Lambda$ to the $y$-axis and $x$-axis. We define the  $P$-edge polynomial of $F$ associated to a $\Lambda$ as
$$
\ff_{P,\Lambda}:=x^{-a_{\Lambda}}y^{-b_{\Lambda}}\sum_{(i,j)\in \Lambda} \bar{f}_{ij} x^i y^j \in \kp_P[x,y].
$$
where $\bar{f}_{ij}\in \kp_P:=\kp[y]/(P)$ stands for the reduction modulo $P$. By construction, the polynomial $\ff_{P,\Lambda}$ is quasi-homogeneous and monic with respect to $x$ and $y$. We say that a series is $P$-convenient if it is not divisible by $P$ or $x$.

\begin{definition} We say that $\F\in \bar{\kp}[[x]][y]$ is non $P$-degenerate if it is $P$-convenient and if both $P$ and all the $P$-edges polynomials of $\F$ are separable with respect to $y$. We say that $\F$ is non degenerate at infinity if  $y^{d_y} \F(1/y)$ is non $y$-degenerate. We say that $\F$ is non degenerate if it is non $P$-degenerate for all irreducible factors $P$ of $\F(0,y)$ and if it is non degenerate at infinity.
\end{definition}

\begin{remark}\label{x=1}
By quasi-homogeneity, we can let $x=1$ for checking separability of the $P$-edge polynomials.
\end{remark}

\begin{remark}
Usually, the notion of non degenerate polynomials in $\kp[[x]][y]$ only allows ramification at the places $y=0$ and $y=\infty$, while we consider here all places of $\mathbb{P}^1_{\kp}$. A notable exception is \citep{Sombra} where the authors use collection of $P$-adic polytopes in order to improve the usual Bernstein-Koushnirenko bound for the number of solutions of a polynomial system with isolated roots.
\end{remark}

\begin{remark}
In zero characteristic, non $y$-degeneracy is equivalent to the most common definition of non degeneracy introduced by Kouchnirenko \citep{Kou}. In positive characteristic, there are several notion of non degeneracy. Kouchnirenko non degeneracy is equivalent to that the edge polynomials are separable with respect to $x$ and $y$. This is the one that allows to generalize the Milnor formula to positive characteristic. Our notion is weaker (for instance $y^3-x^2$ in characteristic $2$). Weak non degeneracy introduced in \citep{BGM}, Section 3 is equivalent to that the edge polynomials are squarefree. This is the one that allows to compute the number of local factors. Our notion is stronger (for instance $y^3-x^2$ in characteristic $3$). 
\end{remark}

\begin{lemma}\label{lemdeg}
Let $P\in \kp[y]$ be separable and irreducible. Then $\F$ is non $P$-degenerate if and only if $\F(y+\alpha)$ is non $y$-degenerate at any roots $\alpha$ of $P$.
\end{lemma}

\begin{proof}
Let $\alpha$ be a root of $P$ and let us write $\F(y+\alpha)=\sum c_{ij} x^i y^j$ for some $c_{ij}\in \bar{\kp}$. A straightforward computation shows that the coefficients in the two expressions are related by
\[\begin{aligned}
c_{ij}=  P'(\alpha)^j f_{ij}(\alpha).
\end{aligned}\]
By Lemma \ref{lempol}, the $y$-edges of $\F(y+\alpha)$ are one-to-one with the $P$-edges of $\F$. Let $\Lambda$ be such an edge. The corresponding $y$-edge polynomial $\ff_{y,\Lambda}$ of $\F(y+\alpha)$ and $P$-edge polyonomial $\ff_{P,\Lambda}$ of $\F$ are related by the formula
$$
\ff_{y,\Lambda}\Big(x,\frac{y}{P'(\alpha)}\Big)=x^{-a_{\Lambda}}y^{-b_{\Lambda}}\sum_{(i,j)\in \Lambda} f_{ij}(\alpha) x^i y^j=\ev_{\alpha}\big(\ff_{P,\Lambda}\big)\in\kp(\alpha)[x,y].
$$
where $\ev_{\alpha}$ is induced be the isomorphism $\kp_P \simeq \kp(\alpha)$ determined by $\alpha$. Since the discriminant of monic polynomials commutes with specialization, we deduce that $\ff_{y,\Lambda}$ is separable with respect to $y$ if and only if $\ff_{P,\Lambda}$ is. 
\end{proof}

\begin{proposition}\label{polytope}
Suppose that $F\in \kp[x,y]$ is non degenerate, then $S=V(\bar{\kp})$.
\end{proposition}

\begin{proof}
Let as usual $F=u\F_1\cdots \F_s$ be the irreducible factorization of $F$ in $\kp[[x]][y]$. By Corollary \ref{cor1}, it's enough to prove that $q_i\le d_x$ for all $i$. 
By point $(2)$ in Lemma \ref{order}, the $q$-invariant of the irreducible factors of $\F_i$ in $\bar{\kp}[[x]][y]$ are all equal to $q_i$.  Hence, there is no less to suppose that $\kp=\bar{\kp}$. In such a case, Hensel lemma implies that we necessarily have $\F_i(0,y)=(y-\alpha_i)^{d_i}$ for some $\alpha_i\in \bar{\kp}$. Since $q_i$ is invariant under the Moebius transformations $F\leftarrow F(x,y+\alpha_i)$ (or $F\leftarrow y^{d_y}F(x,1/y)$ for $\alpha_i=\infty$), we are reduced to estimate $q_i$ when $\F_i(0,0)=0$. Clearly, $q_i$ only depends on those factors $\F_j$ for which $\F_j(0,0)=0$ so that we can suppose that $F\in \kp[x,y]$ is a product of Weierstrass polynomials which by Lemma \ref{lemdeg} is non $y$-degenerate. By the multiplicative property of the resultant, we have
$$
d_iq_i=\val_x(\Disc_y(\F_i))+\sum_{j\ne i} \val_x \Res_y(\F_i,\F_j).
$$
For given two Weirestrass polynomials $\F,\G\in \kp[[x]][y]$, we have the formula
$$
\val_x \Res_y(\F,\G)=(\F,\G)_0
$$
where $(\F,\G)_0$ stands for the intersection multiplicity
$$
(\F,\G)_0:=\dim_{\kp} \frac{\kp[[x,y]]}{(\F,\G)}
$$
of $\F$ and $\G$ at $(0,0)$, see for instance \citep{Tei} p.28 (the proof adapts to the positive characteristic case). Let us denote by 
$$
\Delta_i:=(\rp^+)^2\setminus \N_{\F_i}
$$ 
the complementary set in $(\rp^+)^2$ of the Newton polytope of $\F_i$. Note that $\Delta_i$ is compact since $\F_i$ is convenient by assumption. By Bernstein-Khovanskii-Kouchnirenko theorem, we have the formula
$$
(\F_i,\F_j)_0\ge [\Delta_i,\Delta_j],
$$
\textit{with equality if the product $\F_i\F_j$ is non $y$-degenerate} (the converse holds in characteristic zero). See for instance Corollary 5.6 in \citep{CN} in the case $\kp=\cp$ or \citep{Kou} for any algebraically closed field. Here,
$$
[\Delta_i,\Delta_j]:=\Vol(\Delta_i+\Delta_j)-\Vol(\Delta_i)-\Vol(\Delta_j)
$$ 
stands for the mixed volume of polytopes. In the same way, we have that 
$$
(\F_i,\partial_y\F_i)_0\ge [\Delta_i,\Delta_i]-d_y,
$$
\textit{with equality if $\F_i$ is non degenerate} (see Theorem 5.6 in \citep{CN}, the proof adapts to the positive characteristic case since $\ff_{\Lambda}$ is assumed to be separable with respect to $y$). Since all $\F_i$'s are irreducible, it follows that
$\Delta_i$ is a triangle with vertices $(0,0),(a_i,0),(0,d_i)$ where $a_i=\val_x(\F_i(x,0))$ (with $a_i=0$ and $\Delta_i$ being a segment if $\F_i$ does not depend on $x$). In such a case, we get that 
$$
[\Delta_i,\Delta_j]=\min\{d_i a_j,d_j a_i\},
$$
see \citep{CN}, Section 5. Hence it follows that 
$$
q_id_i\le \sum_{j=1}^{s} d_i a_j-d_y= d_i \val_x F(x,0)-d_y\le d_i d_x-d_y.
$$
The inequality $q_i\le d_x$ follows. 
\end{proof}


Let $F\in \kp[x,y]$ and suppose given the irreducible factorization $$F(0,y)=\prod_{i=1}^t P_i^{n_i}\in \kp[y].$$
For each $i$, we denote by $s_i$ the total number of irreducible rational factors of the $P_i$-edges polynomials of $F$ and by $\ell_i$ the lattice length of the $P_i$-boundary. Note the inequalities
$$
s_i\le \ell_i \le n_i.
$$
In the same way, denote by $s_{\infty}$ the total number of rational irreducible factors of the edge polynomials of $y^{d_y}F(x,1/y)$ and by $\ell_{\infty}$ the lattice length of its Newton boundary. 

\begin{lemma}\label{lnum}
Suppose given  $F\in \kp[x,y]$ primitive with respect to $y$ and $x$. Then, the respective numbers $s$ and $\bar{s}$ of irreducible analytic factors of $F$ over $\kp$ and $\bar{\kp}$ satisfy
$$
s\le s_{F}:=\sum_{i=1}^t s_i+s_{\infty}\quad {\rm and}\quad \bar{s}\le \bar{s}_{F}:=\sum_{i=1}^t \ell_i\deg P_i+\ell_{\infty},
$$
both inequality being equalities if $F$ is non degenerate along the fiber $x=0$. 
\end{lemma}

\begin{proof}
Suppose first $F$ monic w.r.t $y$. By Hensel Lemma, the number of irreducible factors of $F$ in $\kp[[x]][y]$ is the sum  of the numbers of irreducible factors in the local rings $\kp[[x]][y]_{(P_i)}$ when $P_i$ runs over the irreducible factors of $F(0,y)$. Then the assertion for $s$ follows for instance from Chapter 6 in \citep{Cas}. For the absolute case, we consider  the decomposition 
$$
\bar{\kp}[[x]][y]_{(P_i)}=\bigoplus_{P(\alpha)=0}\bar{\kp}[[x,y-\alpha]].
$$ 
By Lemma 4.10 in \citep{BGM}, the number of absolute factors of $F$ in $\bar{\kp}[[x]][y]_{(y-\alpha)}$ is bounded by the lattice length of the Newton boundary of $F(x,y+\alpha)$, which by Lemma \ref{lempol} coincides with $\ell_i$. Moreover, there is equality if $F(x,y+\alpha)$ is non $y$-degenerate, which is the case when $F$ is non degenerate by Lemma \ref{lemdeg}. Summing up over all the roots of $P_i$ and over all $i$, we get the result.  If $F$ is not monic, we conclude in the same way by taking also into account the place at infinity. 
\end{proof}

\begin{remark}
Equalities $s=s_{F}$ and $\bar{s}=\bar{s}_F$ hold in fact with the weaker hypothesis that the edge polynomials are squarefree. 
\end{remark}

\begin{example}
Suppose that 
$$
F(x,y)=y(y^2-2)^{3}-x^{2}(y^{2}-2)+x^5\in \qp[x,y].
$$ 
Then $F(0,y)=y(y^2-2)^{2}$ has two irreducible coprime factors $P_1=y$ and $P_2=y^2-2$. There is a unique $P_1$-edge polynomial  $\ff_{\Lambda,P_1}=-8y+2x^2$, which is obviously separable and irreducible. Hence $s_1=\ell_1=1$. There are two $P_2$-edge polynomials
$$
\ff_1:=\phi y^2 - x^2 \quad {\rm and}\quad \ff_2:=-y + x^3 
$$
where $\phi\in \qp_{P_2}$ stands for the residue class of $y$.  Both polynomials are separable. Since $\phi$ is not a square, $\ff_1$ is irreducible over $\qp_{P_2}$, but has two factors over $\bar{\qp}$. Obviously, $\ff_2$ has exactly $1$ factor over any field. Hence $s_2=1+1=2$ while $\ell_2=2+1=3$. Since there are no points at infinity, we finally get that $F$ has exactly 
$s=s_1+s_2=3$ irreducible factors in $\qp[[x]][y]$ and $\bar{s}=\ell_1+2\ell_2=7$ irreducible factors in $\bar{\qp}[[x]][y]$.
\end{example}

\begin{example}
Suppose that 
$$
F(x,y)=y^6(y^2+1)^{15}-x^{10}(1+y^{21})\in \kp[x,y],
$$ 
where $\kp$ is any field of characteristic $p\ne 2$. Then 
$$
F(0,y)=y^6(y^2+1)^{15}
$$ 
has only two distinct irreducible factors $P_1=y$ and $P_2=y^2+1$. There is a unique $P_1$-edge polynomial  $\ff_{\Lambda,P_1}=y^6-x^{10}$. It is separable with respect to $y$ if and only if $p\ne 2,3$, and we have $s_1=\ell_1=2$ if $p\ne 2$. There is a unique $P_2$-edge polynomial 
$$
\ff_{\Lambda,P_2}=\phi^6 y^{15} -x^{10}(1+\phi^{21})=-y^{15}+(1+\phi)x^{10},
$$
where $\phi$ stands for the residue class of $y$ in $\kp_{P_2}=\kp[y]/(y^2+1)$. This polynomial is separable if and only if $p\ne 2,3,5$ and it is squarefree if and only if $p\ne 2$. For $p\ne 2,3,5$, it has $s_2=2$ irreducible factors over $\kp_{P_2}$ while the lattice length of $\Lambda$ is $\ell_2=5$. Hence, if $p\ne 2,3,5$, then $F$ is non degenerate and has $s=s_1+s_2=4$ factors in $\kp[[x]][y]$ and $\bar{s}=\ell_1+2\ell_2=12$ factors in $\bar{\kp}[[x]][y]$.
\end{example}

\begin{corollary}\label{tpol}
There is a deterministic algorithm that, given $F\in \kp[x,y]$, returns False if $F$ is degenerate and returns the irreducible factors of $F$ over $\kp$ otherwise. The cost of this algorithm is  one univariate factorization of degree at most $d_y$ and 
$$
\sO(d_x d_y s_{F}^{\omega-1}) +\C(d_x)
$$
operations over $\kp$ if $p>d_x(2d_y-1)$ or
$$
\sO(kd_x d_y s_{F}^{\omega-1}) +\sO(k)\C(d_x)
$$ 
operations over $\fp_p$ if $\kp=\fp_{p^k}$. 
\end{corollary}

\begin{proof}
The algorithm is as follows. We first compute the factorization $F(0,y)=\prod_{i=1}^{t} P_i^{n_i}$. For each $i$, we test the separability of $P_i$ and of all the $P_i$-edges polynomials of $F$. This step costs at most $\sO(d_x d_y)$ operations over $\kp$. If $F$ is non degenerate along the fiber $x=0$, then $s=s_F$ by Lemma \ref{lnum} and the separability order satisfies $N\le d_x$ by Proposition \ref{polytope}. Hence we can take $q=d_x$ in Theorem \ref{t2}. Corollary \ref{tpol} follows.
\end{proof}

\begin{remark}
Analytic factorization of non degenerate polynomials may be reduced to Newton iteration in the rings $\kp_P[[x]][y]$ after some translations and monomial change of coordinates. A first estimate leads in that case to $\C(d_x)\subset \sO(s_F d_xd_y^2)$, but we believe we can do better. One of the main obstruction is the difficulty to use a dichotomic multifactor Hensel lifting as in the regular case. 
\end{remark}

\begin{remark}
A cheap pretreatment of $F$ is to look at the fibers $x=0$ and $x=\infty$ (or $y=0$ and $y=\infty$ by reversing the roles played by $x$ and $y$) in order to check if there is a fiber over which $F$ is non degenerate and with the smallest $s_F$ or $\bar{s}_F$ as possible. If we take for instance 
$F(x,y)=y^6(y-1)^{15}-x^{10}+x^{9}y^{21}$, then the fiber $x=0$ leads to $s_F=4$ and $\bar{s}_F=7$ while the fiber $x=\infty$ would lead to $s=\bar{s}=1$ from which we immediately deduce that $F$ is absolutely irreducible, whatever the field is. This kind of strategies based on the relations beewteen the (global) Newton polytope and bivariate factorization have already been considered in the litterature, see for instance \citep{Gao2} or \citep{W2} and the references therein. 
\end{remark}

\section{Conclusion}\label{S8}

When compared to the regular case, a great advantage of working along a critical fiber for factorization is that one has in general less analytic factors to recombine (always strictly less in the absolute case $\kp=\bar{\kp}$). Unfortunately, this might require in general a higher truncated precision, and our main Theorem \ref{t1} seems to suggest a bad worst case complexity. However, the important classes of non degenerate polynomials and locally irreducible polynomials illustrate that we sometimes gain in complexity when compared to the regular case. Might this hold in all generality ? We discuss briefly the two main obstructions for this.

\subsection{Fast analytic factorization ?}

The strength of our approach deeply relies on the complexity of $N$-truncated analytic factorization. This is a crucial problem in singularity theory. One approach in characteristic zero or $>d_y$ is to compute the rational Puiseux series. There are well known algorithms for this, with actual complexity $\mathcal{O}(d^5)$ in terms of the total degree $d$ of $F$ in the case of finite fields \citep{PR}. This complexity is a bit too large for our purpose, but the authors told us that they recently developped  a $\mathcal{O}(d^4)$ complexity algorithm. This is the same complexity as for absolute factorization. In order to fit also in the rational factorization complexity class, we would need $\sO(d^{\omega+1})$ for analytic factorization. This is an open problem.

Note that there exists algorithms for testing local analytic irreducibility of a germ of curve that do not require the use of Puiseux series. The main ingredient is that of \textit{approximate roots} and makes essentialy use of resultants computations \citep{Ab}, \citep{kuo}, \citep{Cos}. Up to our knowledge, no complexity analysis have been done yet. We don't know if approximate roots could lead to fast analytic factorization that would avoid Puiseux series computations.

In characteristic $p<d_y$, the concept of Puiseux series does not make sense and analytic factorization is a much more delicate problem. See for instance \citep{ked} for the generalization of Puiseux series in small positive characteristic.

\subsection{Fast recombinations ?}

Even if we get a fast algorithm for computing the $N$-truncated analytic factors, our brute force analysis of the recombination problem in this paper led to a complexity $\sO(Nd_y r^{\omega-2})\subset \sO(d^{\omega+2})$. However, the  polynomials 
$
\big[ \hat{\F_i}\partial_y \F_i\big]^{N+1}
$ 
we want to recombine have a very particular structure. Namely, they vanish with very high order at some points of the curve $F=0$. An approach could be to write these polynomials in a basis constitued of adjoint polynomials. The number of required equations for solving recombinations (divisibility test by $F$ or closedness of differential forms) would then decreaze. This fact is illutrated in some of our previous works, where we studied the relations beetween resolution of singularities and factorization \citep{W1} or toric genometry and factorization \cite{W2}. Namely we developed factorization algorithms whose linear algebra steps belongs to $\sO(p_a s^{\omega-1})$, with $p_a\le d^2$ being the arithmetic genus of the strict transform of the curve $F=0$ after some sequences of blowing-ups.

\end{document}